\documentclass{gtpart}
\usepackage{amsmath,amssymb,amsthm,stmaryrd}
\usepackage[all]{xy}
\usepackage[usenames,dvipsnames]{xcolor}
\usepackage{tikz}
\usetikzlibrary{shapes,positioning,fit,decorations.pathmorphing}
\usepackage{url}
\usepackage{hyperref}
\usepackage{enumerate}
\usepackage{tensor}
\usepackage{graphicx}
\usepackage{mathtools}
\usetikzlibrary{arrows}

\title{Norm coherence for descent of level structures\\on formal deformations}
\author{Yifei Zhu}
\givenname{Yifei}
\surname{Zhu}

\address{Department of Mathematics\\Southern University of Science and 
         Technology\\Shenzhen\\Guangdong 518055 China}
\email{zhuyf@sustc.edu.cn}

\newtheorem{thm}[equation]{Theorem}
\newtheorem{cor}[equation]{Corollary}
\newtheorem{prop}[equation]{Proposition}
\newtheorem{lem}[equation]{Lemma}
\theoremstyle{definition}
\newtheorem{defn}[equation]{Definition}

\theoremstyle{remark}
\newtheorem{rmk}[equation]{Remark}
\newtheorem{subsec}[equation]{\hskip -.02cm}
\newtheorem{ex}[equation]{Example}

\newcommand{\Aut}{{\rm Aut}}
\newcommand{\BA}{{\mb A}}
\newcommand{\BC}{{\mb C}}
\newcommand{\BD}{{\mb D}}

\newcommand{\bE}{{\text{\bf \em E}}}
\newcommand{\BF}{{\mb F}}
\newcommand{\Bf}{{\text{\bf \em f}}}
\newcommand{\bF}{{\text{\bf \em F}}}

\newcommand{\BL}{{\mb L}}

\newcommand{\BP}{{\mb P}}

\newcommand{\BW}{{\mb W}}
\newcommand{\BZ}{{\mb Z}}
\newcommand{\CA}{{\cal A}}

\newcommand{\CD}{{\cal D}}

\newcommand{\CF}{{\cal F}}
\newcommand{\cF}{\overline {\mb F}}
\newcommand{\CG}{{\cal G}}
\newcommand{\CH}{{\cal H}}
\newcommand{\CI}{{\cal I}}
\newcommand{\CK}{{\cal K}}
\newcommand{\CL}{{\cal L}}
\newcommand{\ck}{\overline k}

\newcommand{\CO}{{\cal O}}
\renewcommand{\co}{{\colon\,}}

\newcommand{\Coord}{{\rm Coord}}
\newcommand{\CP}{{\cal P}}

\newcommand{\Def}{{\rm Def}}

\newcommand{\eg}{{\em e.g.}}

\newcommand{\Eoo}{E$_\infty$}

\newcommand{\FG}{{\rm FG}}
\newcommand{\fm}{{\mf m}}

\newcommand{\Frob}{{\rm Frob}}

\newcommand{\GE}{{\CG_{_\bE}}}
\newcommand{\gE}{{G_{_\bE}}}
\newcommand{\GF}{{\CG_{_\bF}}}
\newcommand{\Gf}{{\CG_{_\Bf}}}

\newcommand{\Hoo}{H$_\infty$}

\newcommand{\ib}{{\em ibid.}}
\newcommand{\ie}{{\em i.e.}}
\newcommand{\iso}{{\rm iso}}
\newcommand{\isog}{{\rm isog}}
\newcommand{\lc}{{\em loc.~cit.}}
\newcommand{\mb}[1]{\mathbb{#1}}
\newcommand{\mf}[1]{\mathfrak{#1}}

\newcommand{\NCoh}{{\rm NCoh}}
\newcommand{\Norm}{{\rm Norm}}
\newcommand{\PE}{{\pi_0\,\bE}}

\newcommand{\pt}{{\rm pt}}
\newcommand{\SE}{{S_{_\bE}}}
\newcommand{\Set}{{\rm Set}}
\newcommand{\SF}{{S_{_\bF}}}
\newcommand{\Sf}{{S_{_\Bf}}}
\newcommand{\Sp}{{\rm Sp}}

\newcommand{\Spf}{{\rm Spf\,}}

\newcommand{\st}{\stepcounter{equation} \tag{\theequation}}

\newcommand{\md}{~~{\rm mod}~}
\newcommand{\ad}{\text{and}}

\newcommand{\id}{{\rm id}}

\newcommand{\tr}{{\rm tr}}
\newcommand{\univ}{{\rm univ}}

\renewcommand{\ker}{{\rm ker}}

\newcommand{\A}{\alpha}
\newcommand{\B}{\beta}

\renewcommand{\d}{\delta}
\newcommand{\f}{\phi}
\newcommand{\G}{\Gamma}
\newcommand{\g}{\gamma}

\newcommand{\ce}{\coloneqq}
\newcommand{\lb}{\llbracket}
\newcommand{\rb}{\rrbracket}

\newcommand{\<}{\langle}
\renewcommand{\>}{\rangle}

\newcommand{\wt}[1]{\textcolor{white}{#1} \!~}

\renewcommand{\tt}[1]{\bf #1~\,}
\newcommand{\stareq}{\stackrel{\star}{=}}

\newcommand{\td}{\tilde{\d}}
\newcommand{\tE}{\widetilde{E\,}\!}
\newcommand{\tF}{\widetilde{F\,}\!}
\newcommand{\tR}{\widetilde{R\,}\!}

\numberwithin{equation}{section}
\renewcommand{\theequation}{\thesection.\arabic{equation}}

\begin{document}

\begin{abstract}
 We give a formulation for descent of level structures 
 on deformations of formal groups, and study the compatibility between 
 the descent and a norm construction.  
 Under this framework, 
 we generalize Ando's construction of \Hoo\,complex orientations for 
 Morava E-theories associated to Honda formal group laws over $\BF_p$.  We 
 show the existence and uniqueness of such an orientation for any Morava 
 E-theory associated to a formal group law over an algebraic extension of 
 $\BF_p$ and, in particular, orientations for a family of elliptic cohomology theories.  
 These orientations correspond to coordinates on deformations of formal 
 groups which are compatible with norm maps along descent.  
\end{abstract}

\maketitle

\section{Introduction}

\begin{subsec}{\tt{Algebraic motivations and statement of results}}
 Let $R$ be a commutative ring with $1$ and let $A$ be an algebra over $R$.  
 Suppose that, as an $R$-module, $A$ is finitely generated and free.  
 The norm of $A$ is a map $A \to R$ which sends $a$ to $\det(a \cdot)$, 
 the determinant of multiplication by $a$ as an $R$-linear transformation on $A$.  
 It is multiplicative but not additive in general.   
 Such norms appear as an important ingredient in various contexts: 
 arithmetic moduli of elliptic curves \cite[\S 1.8, \S 7.7]{KM}, 
 actions of finite group schemes on abelian varieties \cite[\S 12]{AV}, 
 isogenies of one-parameter formal Lie groups over $p$-adic integer rings 
 \cite[\S 1]{Lubin67}.  These norm maps are closely related to construction of 
 {\em quotient objects}.  
 
 It is the purpose here to examine 
 an interaction between norms and the corresponding 
 {\em subobjects}---more precisely, 
 a functorial interaction with chains of subobjects---in the 
 context of Lubin and Tate's formal deformations \cite{LT}.  
 The functoriality amounts to descent of ``level structures'' on deformations 
 (see \S \ref{sec:nc} and \S \ref{sec:ortn}).   
 In this paper, a level structure on a formal group is a choice of 
 finite subgroup scheme, 
 from which we obtain a quotient morphism of formal groups.  
 A norm map between their rings of functions then gets 
 involved in making this quotient morphism into a homomorphism of 
 formal group {\em laws} \eqref{subsec:quotient}.  
 This norm construction is compatible with successive quotient 
 along a chain of subgroups.  
 
 On the other hand, given a {\em deformation} over a $p$-adic integer ring, 
 there is a canonical (\ie, coordinate-free) 
 descent of level structures via Lubin and Tate's universal deformations.  
 Strickland studied the representability of this moduli problem \cite{Str97} 
 so that the descent can be realized as canonical lifts of Frobenius morphisms 
 \eqref{subsec:canlift}.  
 
 Our main result shows the existence and uniqueness of deformations of 
 formal group laws on which the canonical lifts of Frobenius coincide with 
 quotient homomorphisms from the norm construction.  
 We call these deformations {\em norm-coherent} (see \S \ref{sec:nc}, 
 specifically Definition \ref{def:nc}).  
 
 \begin{thm}[{cf.~Theorem \ref{thm:ideal} and Proposition \ref{prop:nc}}]
  \label{thm}
  Let $k$ be an algebraic extension of $\BF_p$, 
  $R$ a complete local ring with residue field containing $k$, 
  $G$ a formal group law over $k$ of finite height, 
  and $F$ a deformation of $G$ to $R$.  
  There exists a unique formal group law $F'$ over $R$, 
  $\star$-isomorphic to $F$, 
  which is norm-coherent.  
  Moreover, when $F$ is a Lubin-Tate universal deformation, 
  $F'$ is functorial under base change of $G / k$, 
  under $k$-isogeny out of $G$, and under $k$-Galois descent.  
 \end{thm}
 
 \begin{rmk}
  In the context of local class field theory via Lubin-Tate theory, 
  Coleman's norm 
  operator is used to compute norm groups \cite[Theorem 11]{Coleman}.  
  Walker observed its similarity to the norm construction above 
  \cite[Chapter 5]{Walker}.  Specifically, he 
  reformulated the norm-coherence condition (for a special case) 
  in terms of a particular way in which Coleman's 
  norm operator acts [\ib, 5.0.10].  
  It would be interesting to understand this connection in view of Theorem \ref{thm}.  
 \end{rmk}
\end{subsec}

\begin{subsec}{\tt{Topological motivations and statement of results}}
 The relevance to topology (and, further, to geometry and mathematical physics) 
 of this functorial interaction between norms and finite formal subgroup schemes 
 lies, for instance, in having highly coherent multiplications for {\em genera}.  
 These are cobordism invariants of manifolds; such multiplications refine the 
 invariants by reflecting symmetries of the geometry.  
 
 A prominent example is the Witten genus for string manifolds, 
 which takes values in 
 the ring of integral modular forms of level 1.  
 Motivated by this, 
 Hopkins and his collaborators developed 
 highly-structured multiplicative {\em orientations} 
 (\ie, genera of {\em families} of manifolds) 
 for elliptic cohomology theories and 
 for a universal theory of topological modular forms 
 \cite{Hopkins95, Hopkins02}.  
 In particular, in \cite{AHS04}, 
 they showed that their sigma orientation $MU\<6\> \to \bE$
 for any elliptic cohomology $\bE$ is \Hoo, 
 a commutativity condition on its multiplicative structure 
 \eqref{subsec:EH}.  
 Their analysis of this \Hoo\!-structure was based on \cite{Ando95, Ando92} 
 where the algebraic condition of norm coherence had made a first appearance.  
 
 Theorem \ref{thm} produces \Hoo\,$MU\<0\>$-orientations for 
 a family of generalized cohomology theories called Morava E-theories 
 \eqref{subsec:E} 
 including those treated by Ando and by Ando, Hopkins, Strickland.  
 
 \begin{thm}[{cf.~Corollary \ref{cor}}]
  \label{thm:cor}
  Let $k$ and $G$ be as in Theorem \ref{thm}.  
  For the form of Morava E-theory associated to $G / k$, 
  there exists a unique $MU\<0\>$-orientation  that is an \Hoo\,map.  
 \end{thm}
 
 \begin{rmk}
  Rezk reminded us that the sigma orientations do not factor through 
  these \Hoo\,$MU\<0\>$-orientations \eqref{subsec:intro}.  
  On the other hand, the coefficient ring of an E-theory (of height $2$) 
  is a certain completion of a ring of modular forms.  
  In \cite{ho}, as a first step, 
  we related its elements to certain quasimodular forms 
  (and to mock modular forms) 
  via Rezk's logarithmic operations; 
  see also \cite[remarks following Theorem 1.29]{fkl'}.  
  Given Theorem \ref{thm:cor}, 
  it would be interesting to have 
  more exotic manifold invariants.  
  In particular, we may investigate an analogue of the modular 
  invariance of a sigma orientation \cite[1.3]{AHS01} in view of the 
  uniqueness above.  
 \end{rmk}
 
 \begin{rmk}
  A natural question is whether there exist \Eoo\,complex orientations for 
  Morava E-theories and, more specifically, whether the orientation in 
  Theorem \ref{thm:cor} rigidifies to be an \Eoo\,map.  
  See \cite{Hopkins-Lawson} for recent progress on \Eoo\,complex 
  orientations, where the norm-coherence condition comes up.  
 \end{rmk}
\end{subsec}

Finally, the expositions in \cite{Rezk-vid-4} and \cite[esp.~\S 4]{fkl'} 
provide some other perspectives.  
See also \cite[esp.~\S 29]{FPFP}.  

\begin{subsec}{\tt{Outline of the paper}}
 In \S \ref{sec:general}, we recall some basic concepts 
 from the theory of formal groups and homotopy theory, 
 particularly quotient of formal groups 
 \eqref{subsec:quotient}, and set their notation.  
 
 In \S \ref{sec:widecat}, following a suggestion of Rezk, 
 we introduce an enlarged category of formal groups (cf.~\cite[\S 4.1]{KM}).  
 This viewpoint will be helpful in clarifying 
 deformations of Frobenius (\ref{pushforward}, \ref{subsec:univ}), 
 descent of level structures (\ref{l_p}, \ref{l_H}), 
 the norm-coherence condition \eqref{x_H}, 
 and functoriality of norm coherence \eqref{subsec:morefun}.  
 
 \S \ref{sec:defofgl} and \S \ref{sec:defofrob} give an 
 account for the theorems of Lubin, Tate \eqref{prop:LT} 
 and of Strickland \eqref{prop:S} on deformations of formal groups.  
 Our formulation follows Rezk's (\eg, in \cite[\S 4]{iph}) 
 but emphasizes formal group {\em laws}.  
 The purpose of these two sections is to provide a detailed exposition as well as 
 a precise setup that is crucial for the notion of norm coherence 
 to follow in desired generality.  
 
 In \S \ref{sec:nc}, we introduce 
 the central notion of this paper, norm coherence 
 (\ref{subsec:ncdef}-\ref{rmk:nc}), building on 
 Ando's framework \cite[\S 2]{Ando95}.  
 We then generalize his theorem and prove Theorem \ref{thm} 
 in \S \ref{sec:nc'}.  Our main results are Proposition 
 \ref{prop:nc} and Theorem \ref{thm:ideal}, the latter stated in 
 a form suggested by Rezk.  
 
 \S \ref{sec:ortn} discusses the corresponding topological result for 
 complex orientations, with \eqref{subsec:intro} introducing further background 
 on work of Ando, of Ando, Hopkins, Strickland, and of Ando, Strickland.  
 In (\ref{subsec:descent}-\ref{rmk:norm}), we compare the setup for our results above 
 with Ando, Hopkins, and Strickland's descent data and norm maps 
 \cite[Parts 1 and 3]{AHS04}.  
 The purpose is to continue the exposition from \S \ref{sec:defofrob} 
 while proving Theorem \ref{thm:cor}.  
\end{subsec}

\begin{subsec}{\tt{Acknowledgements}}
 This paper originated from a referee's comment on the choice of coordinates 
 in one of my earlier works.  I thank the anonymous referee for their 
 demand for precision on specifics.  
 
 I thank Anna Marie Bohmann, Paul Goerss, Fei Han, 
 Michael Hill, Tyler Lawson, 
 Niko Naumann, and Eric Peterson for helpful discussions.  
 I thank Zhen Huan for the quick help with locating a reference.  
 
 I learned most of what I know about norm coherence 
 and related questions from Charles Rezk.  
 A good deal of the theory presented here was developed 
 in discussions with him, including ``norm-coherent.''  
 The term is my choice over the synonym ``Ando'' and 
 it is Matthew Ando who originally discovered this condition in algebra and 
 applied it to topology.  
 
 I thank Eric Peterson for the feedback on a draft of this paper, and for 
 explaining to me the results and methods of his joint work with Nathaniel 
 Stapleton, which gives a different approach to questions considered here.  
\end{subsec}

\begin{subsec}{\tt{Conventions}}
 Unless explicitly noted, we fix a prime $p$ throughout this paper.  
 
 We often omit the symbol $\Spf$\! and simply write $R$ for $\Spf R$ 
 when it appears as a base scheme.  
 In particular, $\B^*$ means base change from $R$ to $S$ along 
 $\B \co R \to S$, understood as $\B \co \Spf S \to \Spf R$.  
 
 We also write $\psi^*$ for the pullback of functions 
 along a morphism $\psi$ of schemes.  
\end{subsec}

\section{General notions}
\label{sec:general}

\begin{subsec}{\tt{Formal groups, coordinates, and formal group laws}}
 \label{subsec:fgl}
 Let $R$ be a complete local ring with residue characteristic 
 $p > 0$.  A {\em formal group $\CG$ over $R$} is a group object in the category 
 of formal $R$-schemes.  In this paper, all formal groups will be commutative, 
 one-dimensional, and affine.  They can be viewed as covariant functors from the 
 category of complete local $R$-algebras (and local homomorphisms) to 
 the category of abelian groups.  
 
 A {\em coordinate $x$ on $\CG$} is a natural isomorphism 
 $\CG \xrightarrow{\sim} \hat{\BA}^1_R$ of functors to pointed sets.  It gives 
 an isomorphism $\G(\CG, \CO_G) \cong R \lb x \rb$ of augmented 
 $R$-algebras, as well as a 
 trivialization of the ideal sheaf $\CI_G(O) = \CO_G(-O)$ 
 of functions on $\CG$ which vanish 
 at the identity section $O$.   
 Here and throughout the paper, we remove 
 the calligraphic effect of the notation for a formal group whenever it appears 
 as a subscript.  We will also simply write $\CO_G$ for the ring of global sections 
 of $\CO_G$, and similarly for other sheaves.  
 
 A {\em (one-dimensional commutative) formal group law $F$ over $R$} is a formal 
 power series in two variables $t_1$ and $t_2$ with coefficients in $R$, often 
 written $t_1 +_F t_2$, which satisfies a set of abelian-group-like axioms.  In 
 particular, the above data of $\CG$ and $x$ determines a formal group law 
 $G = G_x$ such that 
 \[
  x(P_1) \underset{G}{+} x(P_2) = x (P_1 + P_2) 
 \]
 for any $R$-points $P_1$ and $P_2$ on $\CG$ (where we identify an $R$-point 
 on $\hat{\BA}^1_R$ with an element in the maximal ideal of $R$).  Conversely, given a formal group 
 law $F$, it determines a formal group $\CF = \Spf R \lb x_{_F} \rb$ in a 
 similar way.  
\end{subsec}

\begin{subsec}{\tt{Subgroups and isogenies}}
 By {\em (finite) subgroups} of a formal group over $R$, we mean finite 
 flat closed subgroup schemes.  Their points are 
 often defined over an extension $\tR$ of $R$.  
 
 An {\em isogeny} $\psi \co \CG \to \CG'$ over $R$ 
 is a finite flat morphism of formal groups.  
 Along $\psi^*$, $\CO_G$ becomes a free $\CO_{G'}$-module of finite rank $d$, 
 called the {\em degree of $\psi$}.  Since 
 the residue characteristic of $R$ is $p$, $d$ must be a power of $p$.  
 
 Suppose $x$ and $x'$ are coordinates on $\CG$ and $\CG'$.  
 Then $\psi$ induces a homomorphism $G_x \to G'_{x'}$ 
 of formal group laws, \ie, $h(t) \in t \cdot R \lb t \rb$ such that 
 \[
  h \big( t_1 \underset{G_x}{+} t_2 \big) = h(t_1) \underset{G'_{x'}}{+} h(t_2) 
 \]
 In fact, $h(x) = \psi^*(x')$ and sometimes we will abuse notation by writing 
 $\psi$ for $h$.  
 We will also denote this homomorphism by $\psi \co G \to G'$ and say it is an isogeny of 
 degree $d$ (cf.~\cite[1.6]{Lubin67}).  
 By Weierstrass preparation, $h = m n$ with $m \in 
 R[t]$ monic of degree $d$ and $n \in R \lb t \rb$ invertible.  
 
\end{subsec}

\begin{subsec}{\tt{Kernels and quotients}}
 \label{subsec:quotient}
 The notions of subgroups and of isogenies are connected as follows.  
 
 Given $\psi \co \CG \to \CG'$ as above, its {\em kernel $\CK$} is defined by 
 $\CO_K = \CO_G \otimes_{\CO_{G'}} R$, where the tensor product is taken along 
 $\psi^*$ and the augmentation map of $\CO_{G'}$.  
 It is naturally a subgroup of $\CG$ 
 and has degree $d$ as an effective Cartier divisor in $\CG$.  
 
 Conversely, given a subgroup $\CH \subset \CG$ over $\tR$ of degree $p^r$, 
 there is a corresponding isogeny $f_{_H} \co \CG \to \CG / \CH$ 
 defined by an equalizer diagram 
 \[
  \begin{tikzpicture}[baseline={([yshift=-10pt]current bounding box.north)}, scale=.88]
   \node (L) at (0, 0) {$\CO_{G / H}$};
   \node (M) at (2, 0) {$\CO_G$};
   \node (R) at (4, 0) {$\CO_{G \times H}$};
   \draw [->] (L) -- node [above] {$\scriptstyle f_{_H}^*$} (M);
   \draw [transform canvas={yshift=0.5ex}, ->] (M) -- node [above] 
         {$\scriptstyle \mu^*$} (R);
   \draw [transform canvas={yshift=-0.5ex}, ->] (M) -- node [below] 
         {$\scriptstyle \pi^*$} (R);
  \end{tikzpicture}
 \]
 where $\mu, \pi \co \CG \times \CH \to \CG$ are the multiplication, 
 projection maps, and $\CG / \CH$ is naturally a formal group over $\tR$.  Moreover, 
 given a coordinate $x$ on $\CG$, 
 \[
  x_{_H} \ce \Norm_{\,f_{_H}^*}(x) 
 \]
 is a coordinate on $\CG / \CH$, 
 where $\Norm_{\,f_{_H}^*}(x)$ equals the determinant of multiplication by $x$ on $\CO_G$ 
 as a finite free $\CO_{G / H}$-module via $\,f_{_H}^*$.  Explicitly, 
 \begin{equation}
  \label{lubinisog}
  f_{_H}^*(x_{_H}) = \prod_{Q \in \CH(\tR)} \big( x \underset{G}{+} x(Q) \big) 
 \end{equation}
 By writing $\,f_{_H} \co G \to G / H$ as an isogeny of formal group laws, 
 we will always intend the above compatibility between corresponding coordinates.  
 Sometimes we write more specifically 
 \[
  f_{_H}^x \co G_x \to G_x / H \ce (G / H)_{x_{_H}} 
 \]
 
 Note that over the residue field of $R$, \eqref{lubinisog} becomes 
 \begin{equation}
  \label{defofrob}
  f_{_H}^*(x_{_H}) = x^{\,p^r} 
 \end{equation}
 as a formal group over a field of characteristic $p$ has exactly one subgroup of 
 degree $p^r$.  Thus $\,f_{_H}$ is a lift of the relative $p^r$-power Frobenius 
 isogeny.   
 
 For more details, see \cite[\S 1, esp.~Theorems 1.4, 1.5]{Lubin67}, 
 \cite[\S 5, esp.~Theorem 19]{Str97} (cf. Remark \ref{rmk:norm'} below), 
 and \cite[\S\S 2.1-2.2]{Ando95}.  
\end{subsec}

\begin{subsec}{\tt{Complex cobordism}}
 Let $MU\<0\>$ 
 be the Thom spectrum of the tautological (virtual) complex vector bundle over 
 $\BZ \times BU$.  We have $\pi_* MU\<0\> \cong \pi_* MU\,[\B^{\pm 1}]$ 
 with $|\B| = 2$.  More generally, 
 let $MU\<2 k\>$ be the Thom spectrum associated to the $(2 k - 1)$-connected 
 cover $BU\<2 k\> \to \BZ \times BU$.  
 
 The spectrum $MU\<0\>$ is often 
 written $MUP$ or $MP$ for ``periodic'' (as can be seen from its 
 homotopy groups).  In fact, 
 \[
  MU\<0\> = \bigvee_{m \in \BZ} \Sigma^{2 m} MU 
 \]
 so that $\pi_0\,MU\<0\>$ is the ring of cobordism classes 
 of even-dimensional stably almost complex manifolds.  This ring carries 
 the universal formal group law of Lazard by \cite[Theorem 2]{Quillen}.  
 
 The spectrum 
 $MU\<2\> = MU$. The homology of 
 $MU\<2 k\>$ is concentrated in even degrees if $0 \leq k \leq 3$.   
\end{subsec}

\begin{subsec}{\tt{Morava E-theories}}
 \label{subsec:E}
 Let $k$ be a perfect field of characteristic $p$, and $\CG$ be a formal group 
 over $k$ of finite height $n$.  Associated to this data, 
 there is a generalized cohomology theory, called a {\em Morava E-theory 
 (of height $n$ at the prime $p$)}.  
 It is represented by a ring spectrum 
 $\bE = \bE_n$.  The formal scheme $\GE \ce \Spf \bE^0 \BC \BP^\infty$ is naturally 
 a formal group over $\bE^0(\pt) = \PE$.  The above association requires that 
 $\GE$ be a Lubin-Tate universal deformation of $\CG / k$ (see 
 \S \ref{sec:defofgl} below).  We have 
 \[
  \pi_* \hskip .02cm \bE \cong \BW k \lb u_1, \ldots, u_{n - 1} \rb [u^{\pm 1}] 
 \]
 where $|u_i| = 0$ and $|u| = 2$.\footnote{For some purposes, 
 it is convenient to instead 
 have $\BW \hskip .02cm \ck$ or $|u| = -2$ in $\pi_* \hskip .02cm \bE$.}  
 
 Thus a Morava E-theory spectrum is a topological realization of a Lubin-Tate ring.  
 Strickland showed that $\bE^0 B\Sigma_{p^r} / I_\tr$ is 
 a finite free module over $\PE$, where $I_\tr$ is the ideal generated by the images of transfers 
 from proper subgroups of the symmetric group $\Sigma_{p^r}$ on $p^r$ letters.  
 Moreover, this ring classifies degree-$p^r$ subgroups of $\GE$ 
 \cite[Theorem 1.1]{Str98} (see \S \ref{sec:defofrob}).  
 Ando, Hopkins, and Strickland then assembled these into a 
 topological realization of descent data for level structures on $\GE$ in
 \cite[\S 3.2]{AHS04} (see \S \ref{sec:ortn}).  
 
 When $\CG$ is the formal group of a supersingular elliptic curve, 
 its corresponding E-theory (of height $2$) is an {\em elliptic cohomology theory}  
 \cite[Definition 1.2]{AHS01} via the Serre-Tate theorem.  
\end{subsec}

\begin{subsec}{\tt{Complex orientations for Morava E-theories}}
 A {\em complex orientation} for $\bE$ is 
 a coherent choice of Thom class in $\bE$-cohomology for every 
 complex vector bundle.  It amounts to the choice of a single class 
 $\xi \in \widetilde{\bE\,}\!^2 \BC \BP^\infty$ which restricts to $1$ under the 
 composite 
 \[
  \widetilde{\bE\,}\!^2 \BC \BP^\infty \to \widetilde{\bE\,}\!^2 \BC \BP^1 
  = \widetilde{\bE\,}\!^2 S^2 \cong \bE^0(\pt)
 \]
 Given a coordinate $x \in \CO_\gE = \bE^0 \BC \BP^\infty$, 
 as a trivialization for 
 $\CI_{G_{_\bE}}(O) = \widetilde{\bE\,}\!^0 \BC \BP^\infty$, 
 it corresponds to 
 an invertible class $u_x \in \widetilde{\bE\,}\!^0 \BC \BP^1 \cong \pi_2\,\bE$.  
 We then get a complex orientation for $\bE$ from 
 $x \cdot u_x^{-1} \in \widetilde{\bE\,}\!^2 \BC \BP^\infty$.  
 Conversely, we recover a coordinate on $\GE$ from a class $\xi$ above 
 and a generator for $\pi_2\,\bE$.  
 
 An {\em $MU\<0\>$-orientation for $\bE$} is a map $MU\<0\> \to \bE$ 
 of homotopy commutative ring spectra.  
 Consider the natural map 
 \[
  \BC \BP^\infty_+ \to (\BC \BP^\infty)^\CL \to \Sigma^2 MU \to MU\<0\> 
 \]
 where $\CL$ is the tautological line bundle over $\BC \BP^\infty$.  
 Composing with this, each $MU\<0\>$-orientation gives a generator of 
 $\widetilde{\bE\,}\!^0 \BC \BP^\infty$ and thus a coordinate on $\GE$.  
 In fact, the correspondence is a bijection (see 
 \cite[Proposition 1.10\,(ii)]{Ando00} and \cite[Corollary 2.50]{AHS01}).  
\end{subsec}

\begin{subsec}{\tt{\Eoo and \Hoo structures}}
 \label{subsec:EH}
 Let $\Sp$ be a complete and cocomplete category of spectra, 
 indexed over some universe, 
 with an associative and commutative smash product $\wedge$ 
 (\eg, the category of $\BL$-spectra in \cite[Chapter I]{EKMM}).  
 
 An {\em \Eoo\!-ring spectrum} is a commutative monoid in $\Sp$.  
 Equivalently, it is an algebra for the monad $\BD$ on $\Sp$ defined by 
 \[
  \BD(-) \ce \bigvee_{m \geq 0} \BD_m(-) \ce \bigvee_{m \geq 0} (-)^{\wedge m} / \Sigma_m 
 \]
 where $\Sigma_m$ is the symmetric group on $m$ letters acting on the $m$-fold 
 smash product.  
 
 Weaker than being \Eoo, an {\em \Hoo\!-ring spectrum} is a commutative monoid in the homotopy category 
 of $\Sp$.  
 It also has a description as an algebra for the monad which descends from $\BD$ 
 to the homotopy category.  
 In particular, there are power operations $D_m$ 
 on the homotopy groups of such a spectrum (see \cite[Chapter I]{H_infty}).  
 
 Complex cobordism $MU$ and its variants above are \Eoo\!-ring spectra 
 \cite[\S IV.2]{E_infty}.  
 Morava E-theories $\bE$ are also \Eoo\!-ring spectra \cite[Corollary 7.6]{GH}.  
 A morphism of \Eoo\!-ring (or \Hoo\!-ring) spectra is called an 
 {\em \Eoo\!} (or {\em \Hoo}) {\em map}.  
\end{subsec}

\section{Wide categories of formal groups}
\label{sec:widecat}

\begin{subsec}{\tt{The category $\FG$ and its subcategories}}
 Consider $\FG$ whose objects are formal groups 
 \[
  \begin{tikzpicture}[scale=.88]
   \node (T) at (0, 3) {$\CG$};
   \node (B) at (0, 0) {$\Spf k$};
   \draw [->] (T) -- node [right] {$\scriptstyle f$} (B);
  \end{tikzpicture}
 \]
 of finite height 
 over variable base fields of characteristic $p$, 
 and whose morphisms are cartesian squares 
 \begin{equation}
  \label{cartesian}
  \begin{tikzpicture}[baseline={([yshift=-10pt]current bounding box.north)}, scale=.88]
   \node (LT) at (0, 3) {$\CG$};
   \node (RT) at (4, 3) {$\CG'$};
   \node (LB) at (0, 0) {$\Spf k$};
   \node (RB) at (4, 0) {$\Spf k'$};
   \draw [->] (LT) -- node [above] {$\scriptstyle \Psi$} (RT);
   \draw [->] (LT) -- node [left] {$\scriptstyle f$} (LB);
   \draw [->] (RT) -- node [right] {$\scriptstyle f'$} (RB);
   \draw [->] (LB) -- node [above] {$\scriptstyle \B$} (RB);
  \end{tikzpicture}
 \end{equation}
 \ie, commutative squares such that the induced morphism of $k$-schemes 
 \begin{equation}
  \label{hom}
  \CG \xrightarrow{(\Psi, ~ f)} \CG' \underset{k'}{\times} k 
 \end{equation}
 is a homomorphism of formal groups over $k$.  
 We also have subcategories $\FG_\isog$ and $\FG_\iso$ when \eqref{hom} is 
 restricted to be an isogeny or isomorphism.  Write $\FG(k)$, $\FG_\isog(k)$, 
 and $\FG_\iso(k)$ for the 
 subcategories where the base field is fixed and $\B = \id$ in \eqref{cartesian}.  

 We think of $\FG$, $\FG_\isog$, $\FG_\iso$ as ``wide'' categories 
 given the factorization 
 \begin{equation}
  \label{widemor}
  \begin{tikzpicture}[baseline={([yshift=-10pt]current bounding box.north)}, scale=.88]
   \node (LT) at (0, 3) {$\CG$};
   \node (MT) at (4, 3) {$\CG' \times_{k'} k$};
   \node (RT) at (8, 3) {$\CG'$};
   \node (LB) at (0, 0) {$\Spf k$};
   \node (MB) at (4, 0) {$\Spf k$};
   \node (RB) at (8, 0) {$\Spf k'$};
   \draw [->] (LT) -- node [above] {$\scriptstyle$} (MT);
   \draw [->] (MT) -- (RT);
   \draw [->] (LT) -- node [left] {$\scriptstyle f$} (LB);
   \draw [->] (MT) -- (MB);
   \draw [->] (RT) -- node [right] {$\scriptstyle f'$} (RB);
   \draw [=] (LB) -- (MB);
   \draw [->] (MB) -- node [above] {$\scriptstyle \B$} (RB);
   \node at (4.5, 2.5) {$\lrcorner$}; 
  \end{tikzpicture}
 \end{equation}
\end{subsec}

\begin{ex}
 \label{ex:frob}
 For our purpose, a key example of morphisms in $\FG$ is the following, 
 where $\sigma$ is the absolute $p$-power Frobenius and $\Frob$ is the relative 
 one.  
 \begin{equation}
  \label{frob}
  \begin{tikzpicture}[baseline={([yshift=-10pt]current bounding box.north)}, scale=.88]
   \node (LT) at (0, 3) {$\CG$};
   \node (MT) at (4, 3) {$\CG^{(p)}$};
   \node (RT) at (8, 3) {$\CG$};
   \node (LB) at (0, 0) {$\Spf k$};
   \node (MB) at (4, 0) {$\Spf k$};
   \node (RB) at (8, 0) {$\Spf k$};
   \draw [->] (LT) -- node [above] {$\scriptstyle \Frob$} (MT);
   \draw [->] (LT) to [out = 30, in = 150]  node [above] {$\scriptstyle \sigma$} (RT);
   \draw [->] (MT) -- (RT);
   \draw [->] (LT) -- node [left] {$\scriptstyle f$} (LB);
   \draw [->] (MT) -- (MB);
   \draw [->] (RT) -- node [right] {$\scriptstyle f$} (RB);
   \draw [=] (LB) -- (MB);
   \draw [->] (MB) -- node [above] {$\scriptstyle \sigma$} (RB);
   \node at (4.5, 2.5) {$\lrcorner$}; 
  \end{tikzpicture}
 \end{equation}
 This is an endomorphism in $\FG_\isog$ on the object $\CG / k$.  
 Denote it by $\Phi$.  
 It is not a morphism in $\FG_\isog(k)$.  
 The composite $\Phi^r$ corresponds to the $p^r$-power Frobenius.  
\end{ex}

\begin{subsec}{\tt{Canonical factorization of $\Phi^r$ along an isogeny}}
 Given any $\psi \co \CG \to \CG'$ in $\FG_\isog(k)$, 
 necessarily of degree $p^r$ for some $r \geq 0$, 
 there is a unique factorization in $\FG_\isog$ of $\Phi^r$ 
 along $\psi$ as follows, where $\Phi^r = \Lambda_\psi \circ \psi$ 
 with $\Lambda_\psi$ in $\FG_\iso$.  
 \begin{equation}
  \label{factor}
  \begin{tikzpicture}[baseline={([yshift=-10pt]current bounding box.north)}, scale=.88]
   \node (LT) at (0, 3) {$\CG$};
   \node (MLT) at (4, 3) {$\CG'$};
   \node (MRT) at (8, 3) {$\CG^{(p^r)}$};
   \node (RT) at (12, 3) {$\CG$};
   \node (LB) at (0, 0) {$\Spf k$};
   \node (MLB) at (4, 0) {$\Spf k$};
   \node (MRB) at (8, 0) {$\Spf k$};
   \node (RB) at (12, 0) {$\Spf k$};
   \draw [->] (LT) -- node [above] {$\scriptstyle \psi$} (MLT);
   \draw [->] (MLT) -- node [above] {$\scriptstyle \sim$} (MRT);
   \draw [->] (MRT) -- (RT);
   \draw [->] (LT) -- (LB);
   \draw [->] (MLT) -- (MLB);
   \draw [->] (MRT) -- (MRB);
   \draw [->] (RT) -- (RB);
   \draw [double] (LB) -- (MLB);
   \draw [double] (MLB) -- (MRB);
   \draw [->] (MRB) -- node [above] {$\scriptstyle \sigma^r$} (RB);
   \draw [->] (MLT) to [out = 30, in = 150] node [above] {$\scriptstyle \Lambda_\psi$} (RT);
   \node at (8.5, 2.5) {$\lrcorner$}; 
  \end{tikzpicture}
 \end{equation}
 Correspondingly, between rings of functions on the formal groups, we have 
 \begin{equation}
  \label{factor'}
  \begin{tikzpicture}[baseline={([yshift=-10pt]current bounding box.north)}, scale=.88]
   \node (L) at (0, 0) {$\CO_G$};
   \node (ML) at (4, 0) {$\CO_{G'}$};
   \node (R) at (12, 0) {$\CO_G$};
   \draw [left hook->] (ML) -- node [above] {$\scriptstyle \psi^*$} (L);
   \draw [->] (R) -- node [above] {$\scriptstyle \Lambda_\psi^*$} (ML);
   \draw [->] (R) to [out = -165, in = -15] node [below] {$\scriptstyle \sigma^r$} (L);
  \end{tikzpicture}
 \end{equation}
 When $\psi$ is the relative Frobenius, the map $\Lambda_\psi^*$ has a simple description below.  

 \begin{lem}
 \label{lem:norm}
  Let $\psi = \Frob^r$ in \eqref{factor}.  Then $\Lambda_\psi^*$ coincides with 
  the norm of 
  $\CO_G$ as a finite free module over $\CO_{G'}$ along $\psi^*$; that is, 
  given any $a \in \CO_G$, $\Lambda_\psi^*(a)$ equals the determinant of 
  multiplication by $a$ as an $\CO_{G'}$-linear transformation on $\CO_G$.  
 \end{lem}
 \begin{proof}
  Let $\CO_G = k \lb x \rb$ and $\CO_{G'} = k \lb y \rb$ 
  with $\psi^*(y) = x^{\,p^r}$.  
  Write the norm map as $\Norm_{\psi^*} \co \CO_G \to \CO_{G'}$.  
  We have 
  \begin{equation}
   \label{norm}
   \Norm_{\psi^*}(x) = \prod_{i = 1}^{p^r} x_i = (-1)^{\,p^r + 1} y = y 
  \end{equation}
  and $\Norm_{\psi^*}(c) = c^{\,p^r}$, where $c \in k$ and $x_i$ are the roots of the minimal polynomial 
  of $x$ over $\CO_{G'}$.  
  Note that in characteristic $p$, the norm map is additive and hence 
  a local homomorphism.  Thus composing with the $k$-linear map 
  $\psi^*$, it becomes the absolute $p^r$-power Frobenius $\sigma^r$ as follows, 
  where $h^{(p^r)}$ is the series obtained by 
  twisting the coefficients of $h$ with the $p^r$-power Frobenius.  
  \begin{equation}
   \label{norm'}
   \begin{tikzpicture}[baseline={([yshift=-10pt]current bounding box.north)}, scale=.88]
    \node (LT) at (0, .7) {$k\lb x\rb$};
    \node (MLT) at (3.5, .7) {$k\lb y\rb$};
    \node (RT) at (10.5, .7) {$k\lb x\rb$};
    \draw [->] (MLT) -- node [above] {$\scriptstyle \psi^*$} (LT);
    \draw [->] (RT) -- node [above] {$\scriptstyle \Norm_{\psi^*}$} (MLT);
    \node (LB) at (0, 0) {$h^{(p^r)}(x^{\,p^r})$};
    \node (MLB) at (3.5, 0) {$h^{(p^r)}(y)$};
    \node (RB) at (10.5, 0) {$h(x)$};
    \draw [|->] (MLB) -- (LB);
    \draw [|->] (RB) -- (MLB);
   \end{tikzpicture}
  \end{equation}
  The claim then follows by the uniqueness of the factorization \eqref{factor}.  
 \end{proof}
\end{subsec}

\section{Deformations of formal group laws}
\label{sec:defofgl}

\begin{subsec}{\tt{Set-up}}
 \label{subsec:setup}
 Let $k$ be a field of characteristic $p > 0$, and $G$ be a formal group law over 
 $k$ of height $n < \infty$.  Let $R$ be a complete local ring 
 with maximal ideal $\fm$ and residue field $R / \fm \supset k$, and 
 let $\pi \co R \to R / \fm$ be the natural projection.  
\end{subsec}

\begin{subsec}{\tt{Deformations and deformation structures}}
 \label{subsec:ds}
 A {\em deformation of $G$ to $R$} is a triple $(F, i, \eta)$ consisting of a 
 formal group law $F$ over $R$, an inclusion $i \co k \hookrightarrow R / \fm$ of 
 fields, and an isomorphism $\eta \co \pi^* F \xrightarrow{\sim} i^* G$ of formal 
 group laws over $R / \fm$, as in the following commutative diagram.  
 The leftmost 
 column is supposed to ``deform'' or ``thicken'' the rightmost column.  
 \[
  \begin{tikzpicture}[scale=.88]
   \node (AA) at (3, 6) {$F$};
   \node (AB) at (6, 6) {$\pi^* F$};
   \node (AC) at (9, 6) {$i^* G$};
   \node (AD) at (12, 6) {$G$};
   \draw [->] (AB) -- (AA);
   \draw [->] (AB) -- node [above] {$\scriptstyle \eta$} (AC);
   \draw [->] (AC) -- (AD);
   \node (AA') at (3, 3) {$\Spf R$};
   \node (AB') at (6, 3) {$\Spf R / \fm$};
   \node (AC') at (9, 3) {$\Spf R / \fm$};
   \node (AD') at (12, 3) {$\Spf k$};
   \draw [->] (AB') -- node [above] {$\scriptstyle \pi$} (AA');
   \draw [double] (AB') -- (AC');
   \draw [->] (AC') -- node [above] {$\scriptstyle i$} (AD');
   \draw [->] (AA) -- (AA');
   \draw [->] (AB) -- (AB');
   \draw [->] (AC) -- (AC');
   \draw [->] (AD) -- (AD');
   \node at (5.5, 5.5) {$\llcorner$};
   \node at (9.5, 5.5) {$\lrcorner$};
  \end{tikzpicture}
 \]
 We call the pair $(i, \eta)$ a {\em deformation structure attached to $F$ 
 with respect to $G / k$}, and 
 may simply call $F$ a deformation of $G$ to $R$ if its deformation structure is 
 understood.  We also call the corresponding formal group $\CF$ a deformation of 
 the formal group $\CG$ to $R$.  
\end{subsec}

\begin{subsec}{\tt{Base change of deformation structures}}
 \label{subsec:bc}
 Let $\B \co R \to S$ be a local homomorphism, and $\bar{\B} \co R / \fm \to 
 S / {\mf n}$ be the induced map between residue fields.  Given a deformation 
 $(F, i, \eta)$ of $G$ to $R$, there is a deformation $(\B^* F, i', \eta)$ to $S$ 
 by base change along $\B$ such that the following diagram commutes.  
 We write $\B^* (i, \eta) \ce (i', \eta) = (\bar{\B} \circ i, \eta)$.  
 \[
  \begin{tikzpicture}[scale=.88]
   \node (AA) at (3, 6) {$\B^* F$};
   \node (AB) at (6, 6) {$\bar{\B}^* \pi^* F$};
   \node (AC) at (9, 6) {$i'^* G$};
   \node (AD) at (12, 6) {$G$};
   \node (BA) at (1.5, 4) {$F$};
   \node (BB) at (4.5, 4) {$\pi^* F$};
   \node (BC) at (7.5, 4) {$i^* G$};
   \node (BD) at (10.5, 4) {$G$};
   \draw [->] (AB) -- (AA);
   \draw [->] (AB) -- node [above] {$\scriptstyle \eta$} (AC);
   \draw [->] (AC) -- (AD);
   \draw [->] (BB) -- (BA);
   \draw [->] (BB) -- node [above] {$\scriptstyle \hskip .3cm \eta$} (BC);
   \draw [->] (BC) -- (BD);
   \draw [->] (AA) -- (BA);
   \draw [->] (AB) -- (BB);
   \draw [->] (AC) -- (BC);
   \draw [double] (AD) -- (BD);
   \node (AA') at (3, 3) {$\Spf S$};
   \node (AB') at (6, 3) {$\Spf S / {\mf n}$};
   \node (AC') at (9, 3) {$\Spf S / {\mf n}$};
   \node (AD') at (12, 3) {$\Spf k$};
   \node (BA') at (1.5, 1) {$\Spf R$};
   \node (BB') at (4.5, 1) {$\Spf R / \fm$};
   \node (BC') at (7.5, 1) {$\Spf R / \fm$};
   \node (BD') at (10.5, 1) {$\Spf k$};
   \draw [->] (AB') -- (AA');
   \draw [double] (AB') -- (AC');
   \draw [->] (AC') -- node [above] {$\scriptstyle ~~ i'$} (AD');
   \draw [->] (BB') -- node [above] {$\scriptstyle \pi$} (BA');
   \draw [double] (BB') -- (BC');
   \draw [->] (BC') -- node [above] {$\scriptstyle i$} (BD');
   \draw [->] (AA') -- node [left] {$\scriptstyle \B$} (BA');
   \draw [->] (AB') -- node [left] {$\scriptstyle \bar{\B}$} (BB');
   \draw [->] (AC') -- (BC');
   \draw [double] (AD') -- (BD');
   \draw [->] (AA) -- (AA');
   \draw [->] (AB) -- (AB');
   \draw [->] (AC) -- (AC');
   \draw [->] (AD) -- (AD');
   \draw [->] (BA) -- (BA');
   \draw [->] (BB) -- (BB');
   \draw [->] (BC) -- (BC');
   \draw [->] (BD) -- (BD');
  \end{tikzpicture}
 \]
\end{subsec}

\begin{subsec}{\tt{$\star$-isomorphisms}}
 Let $(F, i, \eta)$ and $(F', i', \eta')$ be deformations of $G / k$ to $R$.  A {\em 
 $\star$-isomorphism} $(F, i, \eta) \to (F', i', \eta')$ consists of an equality 
 $i = i'$ and an isomorphism $\psi \co F \xrightarrow{\sim} F'$ of formal group laws 
 over $R$ such that $\eta' \circ \pi^* \psi = \eta$, as in the following 
 commutative diagram.  
 \begin{equation}
  \label{stariso}
  \begin{tikzpicture}[baseline={([yshift=-10pt]current bounding box.north)}, scale=.88]
   \node (AA) at (3, 6) {$F$};
   \node (AB) at (6, 6) {$\pi^* F$};
   \node (AC) at (9, 6) {$i^* G$};
   \node (AD) at (12, 6) {$G$};
   \node (BA) at (1.5, 4) {$F'$};
   \node (BB) at (4.5, 4) {$\pi^* F'$};
   \node (BC) at (7.5, 4) {$i^* G$};
   \node (BD) at (10.5, 4) {$G$};
   \draw [->] (AB) -- (AA);
   \draw [->] (AB) -- node [above] {$\scriptstyle \eta$} (AC);
   \draw [->] (AC) -- (AD);
   \draw [->] (BB) -- (BA);
   \draw [->] (BB) -- node [above] {$\scriptstyle \hskip .3cm \eta'$} (BC);
   \draw [->] (BC) -- (BD);
   \draw [->] (AA) -- node [left] {$\scriptstyle \psi$} (BA);
   \draw [->] (AB) -- node [left] {$\scriptstyle \pi^* \psi$} (BB);
   \draw [double] (AC) -- (BC);
   \draw [double] (AD) -- (BD);
   \node (AA') at (3, 3) {$\Spf R$};
   \node (AB') at (6, 3) {$\Spf R / \fm$};
   \node (AC') at (9, 3) {$\Spf R / \fm$};
   \node (AD') at (12, 3) {$\Spf k$};
   \node (BA') at (1.5, 1) {$\Spf R$};
   \node (BB') at (4.5, 1) {$\Spf R / \fm$};
   \node (BC') at (7.5, 1) {$\Spf R / \fm$};
   \node (BD') at (10.5, 1) {$\Spf k$};
   \draw [->] (AB') -- node [above] {$\scriptstyle \pi$} (AA');
   \draw [double] (AB') -- (AC');
   \draw [->] (AC') -- node [above] {$\scriptstyle i$} (AD');
   \draw [->] (BB') -- node [above] {$\scriptstyle \pi$} (BA');
   \draw [double] (BB') -- (BC');
   \draw [->] (BC') -- node [above] {$\scriptstyle i$} (BD');
   \draw [double] (AA') -- (BA');
   \draw [double] (AB') -- (BB');
   \draw [double] (AC') -- (BC');
   \draw [double] (AD') -- (BD');
   \draw [->] (AA) -- (AA');
   \draw [->] (AB) -- (AB');
   \draw [->] (AC) -- (AC');
   \draw [->] (AD) -- (AD');
   \draw [->] (BA) -- (BA');
   \draw [->] (BB) -- (BB');
   \draw [->] (BC) -- (BC');
   \draw [->] (BD) -- (BD');
  \end{tikzpicture}
 \end{equation}
 Continuing with the above definition, we simply call $\psi \co F \to F'$ a {\em 
 $\star$-isomorphism} if in addition $\eta = 
 \eta'$ so that $\pi^* \psi = \id$.  We use the symbol $\stareq$ for 
 this equivalence relation.  Clearly it is preserved under base change.  
\end{subsec}

\begin{prop}[{cf.~\cite[Proposition 4.2]{iph}}]
 \label{prop:LT}
 Let $k$, $G$, $R$ be as in \eqref{subsec:setup} and fix $G / k$.  
 Then the functor 
 \[
  R \mapsto \{\text{$\star$-isomorphism classes of deformations $(F, i, \eta)$ 
  of $G$ to $R$}\} 
 \]
 from the category of complete local rings with residue field 
 containing $k$ to the category of sets is co-represented by the ring $E_n \ce 
 \BW k \lb u_1, \ldots, u_{n - 1} \rb$.  Explicitly, there is a (by no means 
 unique)\,\footnote{See \eqref{subsec:unique} below.} deformation $(F_\univ, \id, \id)$ to $E_n$ satisfying the following 
 universal property.  Given any deformation $(F, i, \eta)$ of $G$ to $R$, there 
 is a unique local homomorphism 
 \[
  \A \co E_n \to R 
 \]
 such that it reduces to $i \co k = E_n / I \hookrightarrow R / \fm$, with $I$ 
 and $\fm$ the maximal ideals, and such that there is a unique 
 $\star$-isomorphism 
 \begin{equation}
  \label{wlog}
  (F, i, \eta) \to (\A^* F_\univ, i, \id) 
 \end{equation}
\end{prop}
\begin{proof}
 Let $\pi \co R \to R / \fm$ and $\rho \co E_n \to E_n / I$ be the natural 
 projections.  When $\eta$ is allowed to be the identity only, this is 
 \cite[Theorem 3.1]{LT} (cf.~\cite[Theorem 2.3.1]{Ando95}).  More generally, the 
 universal property claims that the following diagram commutes, where $\A$ and 
 $\psi$ are both unique.  
 \begin{equation}
  \label{LT}
  \begin{tikzpicture}[baseline={([yshift=-10pt]current bounding box.north)}, scale=.88]
   \node (AA) at (3, 6) {$F$};
   \node (AB) at (6, 6) {$\pi^* F$};
   \node (AC) at (9, 6) {$i^* G$};
   \node (AD) at (12, 6) {$G$};
   \node (BA) at (1.5, 4) {$\A^* F_\univ$};
   \node (BB) at (4.5, 4) {$\pi^* \A^* F_\univ$};
   \node (BC) at (7.5, 4) {$i^* G$};
   \node (BD) at (10.5, 4) {$G$};
   \node (CA) at (0, 2) {$F_\univ$};
   \node (CB) at (3, 2) {$\rho^* F_\univ$};
   \node (CC) at (6, 2) {$G$};
   \node (CD) at (9, 2) {$G$};
   \draw [->] (AB) -- (AA);
   \draw [->] (AB) -- node [above] {$\scriptstyle \eta$} (AC);
   \draw [->] (AC) -- (AD);
   \draw [->] (BB) -- (BA);
   \draw [->] (BC) -- (BD);
   \draw [->] (AA) -- node [left] {$\scriptstyle \psi$} (BA);
   \draw [->] (BA) -- (CA);
   \draw [->] (AB) -- node [left] {$\scriptstyle \pi^* \psi$} (BB);
   \draw [->] (BB) -- (CB);
   \draw [double] (AC) -- (BC);
   \draw [->] (BC) -- (CC);
   \draw [double] (AD) -- (BD);
   \node (AA') at (3, 3) {$\Spf R$};
   \node (AB') at (6, 3) {$\Spf R / \fm$};
   \node (AC') at (9, 3) {$\Spf R / \fm$};
   \node (AD') at (12, 3) {$\Spf k$};
   \node (BA') at (1.5, 1) {$\Spf R$};
   \node (BB') at (4.5, 1) {$\Spf R / \fm$};
   \node (BC') at (7.5, 1) {$\Spf R / \fm$};
   \node (BD') at (10.5, 1) {$\Spf k$};
   \node (CA') at (0, -1) {$\Spf E_n$};
   \node (CB') at (3, -1) {$\Spf E_n / I$};
   \node (CC') at (6, -1) {$\Spf E_n / I$};
   \node (CD') at (9, -1) {$\Spf k$};
   \draw [->] (AB') -- node [above] {$\scriptstyle \pi$} (AA');
   \draw [double] (AB') -- (AC');
   \draw [->] (AC') -- node [above] {$\scriptstyle i$} (AD');
   \draw [->] (BB') -- node [above] {$\scriptstyle \pi$} (BA');
   \draw [double] (BB') -- (BC');
   \draw [->] (BC') -- node [above] {$\scriptstyle i$} (BD');
   \draw [->] (CB') -- node [above] {$\scriptstyle \rho$} (CA');
   \draw [double] (CB') -- (CC');
   \draw [double] (CC') -- (CD');
   \draw [double] (AA') -- (BA');
   \draw [->] (BA') -- node [left] {$\scriptstyle \A$} (CA');
   \draw [double] (AB') -- (BB');
   \draw [->] (BB') -- node [left] {$\scriptstyle i$} (CB');
   \draw [double] (AC') -- (BC');
   \draw [->] (BC') -- node [left] {$\scriptstyle i$} (CC');
   \draw [double] (AD') -- (BD');
   \draw [double] (BD') -- (CD');
   \draw [->] (AA) -- (AA');
   \draw [->] (AB) -- (AB');
   \draw [->] (AC) -- (AC');
   \draw [->] (AD) -- (AD');
   \draw [->] (BA) -- (BA');
   \draw [->] (BB) -- (BB');
   \draw [->] (BC) -- (BC');
   \draw [->] (BD) -- (BD');
   \draw [->] (CA) -- (CA');
   \draw [->] (CB) -- (CB');
   \draw [->] (CC) -- (CC');
   \draw [->] (CD) -- (CD');
   \draw [double] (BB) -- (BC);
   \draw [->] (CB) -- (CA);
   \draw [double] (CB) -- (CC);
   \draw [double] (CC) -- (CD);
   \draw [double] (BD) -- (CD);
  \end{tikzpicture}
 \end{equation}

 To show this, we refine a half of the diagram as follows, omitting the other half.  
 \begin{equation}
  \label{refine}
  \hskip 1.3cm
  \begin{tikzpicture}[baseline={([yshift=-10pt]current bounding box.north)}, scale=.88]
   \node (AA) at (3, 6) {$F$};
   \node (AB) at (6, 6) {$\pi^* F$};
   \node (AC) at (9, 6) {$i^* G$};
   \node (AD) at (12, 6) {$G$};
   \node (DA) at (2.25, 5) {$\tF$};
   \node (DB) at (5.25, 5) {$i^* G$};
   \node (DB') at (5.25, 2) {$\Spf R / \fm$};
   \draw [->] (DB) -- (DB');
   \node (DC) at (8.25, 5) {$i^* G$};
   \node (DD) at (11.25, 5) {$G$};
   \node (BA) at (1.5, 4) {$\A^* F_\univ$};
   \node (BB) at (4.5, 4) {$\pi^* \A^* F_\univ$};
   \node (BC) at (7.5, 4) {$i^* G$};
   \node (BD) at (10.5, 4) {$G$};
   \draw [->] (AB) -- (AA);
   \draw [->] (AB) -- node [above] {$\scriptstyle \eta$} (AC);
   \draw [->] (AC) -- (AD);
   \draw [->] (DB) -- (DA);
   \draw [->] (DC) -- (DD);
   \draw [->] (BB) -- (BA);
   \draw [->] (BC) -- (BD);
   \draw [->] (AA) -- node [left] {$\scriptstyle \f$} (DA);
   \draw [->] (DA) -- node [left] {$\scriptstyle \g$} (BA);
   \draw [->] (AB) -- node [left] {$\scriptstyle \eta$} (DB);
   \draw [double] (DB) -- (BB);
   \draw [double] (AC) -- (DC);
   \draw [double] (DC) -- (BC);
   \draw [double] (AD) -- (DD);
   \draw [double] (DD) -- (BD);
   \node (AA') at (3, 3) {$\Spf R$};
   \node (AB') at (6, 3) {$\Spf R / \fm$};
   \node (AC') at (9, 3) {$\Spf R / \fm$};
   \node (AD') at (12, 3) {$\Spf k$};
   \node (DA') at (2.25, 2) {$\Spf R$};
   \node (DC') at (8.25, 2) {$\Spf R / \fm$};
   \node (DD') at (11.25, 2) {$\Spf k$};
   \node (BA') at (1.5, 1) {$\Spf R$};
   \node (BB') at (4.5, 1) {$\Spf R / \fm$};
   \node (BC') at (7.5, 1) {$\Spf R / \fm$};
   \node (BD') at (10.5, 1) {$\Spf k$};
   \draw [->] (AB') -- node [above] {$\scriptstyle \pi$} (AA');
   \draw [double] (AB') -- (AC');
   \draw [->] (AC') -- node [above] {$\scriptstyle i$} (AD');
   \draw [->] (DB') -- node [above] {$\scriptstyle \pi$} (DA');
   \draw [double] (DB') -- (DC');
   \draw [->] (DC') -- node [above] {$\scriptstyle i$} (DD');
   \draw [->] (BB') -- node [above] {$\scriptstyle \pi$} (BA');
   \draw [double] (BB') -- (BC');
   \draw [->] (BC') -- node [above] {$\scriptstyle i$} (BD');
   \draw [double] (AA') -- (DA');
   \draw [double] (DA') -- (BA');
   \draw [double] (AB') -- (DB');
   \draw [double] (DB') -- (BB');
   \draw [double] (AC') -- (DC');
   \draw [double] (DC') -- (BC');
   \draw [double] (AD') -- (DD');
   \draw [double] (DD') -- (BD');
   \draw [->] (AA) -- (AA');
   \draw [->] (AB) -- (AB');
   \draw [->] (AC) -- (AC');
   \draw [->] (AD) -- (AD');
   \draw [->] (DA) -- (DA');
   \draw [->] (DC) -- (DC');
   \draw [->] (DD) -- (DD');
   \draw [->] (BA) -- (BA');
   \draw [->] (BB) -- (BB');
   \draw [->] (BC) -- (BC');
   \draw [->] (BD) -- (BD');
   \draw [double] (DB) -- (DC);
   \draw [double] (BB) -- (BC);
  \end{tikzpicture}
 \end{equation}
 Here $\phi \co F \to \tF$ over $R$ is any isomorphism lifting 
 $\eta$,\footnote{Such lifts always exist because 
 the ring co-representing (strict) isomorphisms 
 between formal group laws over commutative rings is free polynomial.  
 They are in fact unique by the uniqueness in \cite[Theorem 3.1]{LT}.} 
 so that $\tF$ 
 has deformation structure $(i, \id)$.  By \cite[Theorem 3.1]{LT}, there is a 
 unique local homomorphism $\A \co E_n \to R$ such that it reduces to $i \co k 
 = E_n / I \hookrightarrow R / \fm$ and such that there is a unique 
 $\star$-isomorphism $\g \co \tF \to \A^* F_\univ$.  Thus \eqref{refine} 
 commutes and, consequently, so does \eqref{LT} if we take $\psi =  \g \circ \f$.  
 
 Now, to show the uniqueness, suppose that $\A'$ and $\psi'$ fit into \eqref{LT} in place of $\A$ and $\psi$.  
 Then $\A^* F_\univ$ and $\A'^* F_\univ$ are in the same 
 $\star$-isomorphism class via $\psi' \circ \psi^{-1}$, so by the uniqueness \lc~we 
 have $\A = \A'$.  Moreover, $\psi' \circ \f^{-1} = \g$ so that 
 $\psi' = \g \circ \f = \psi$.  
\end{proof}

\begin{subsec}{\tt{Non-uniqueness of $F_\univ$}}
 \label{subsec:unique}
 There can be 
 $F_\univ / E_n$ and $F_\univ' / E_n$, both satisfying the universal property.  
 Namely, there exists a unique $\A_\univ \co E_n \to E_n$ with a unique 
 $F_\univ' \stareq \A_\univ^* F_\univ$, and 
 there exists a unique $\A_\univ' \co E_n \to E_n$ with a unique 
 $F_\univ \stareq \A_\univ'^* F_\univ'$.  Moreover, we have 
 $\A_\univ' \circ \A_\univ = \A_\univ \circ \A_\univ' = \id$.  

 Suppose that {\em a priori} we know $F_\univ \stareq F_\univ'$.  
 Then $\A_\univ = \A_\univ' = \id$ and this $\star$-isomorphism is unique.  
 Thus the classifying maps $\A \co E_n \to R$ for $F / R$ are independent of 
 the choice between $F_\univ$ and $F_\univ'$.  
\end{subsec}

\section{Deformations of Frobenius}
\label{sec:defofrob}

The flexibility of having an isomorphism $\eta$ in a deformation 
of a formal group law 
buys us a notion of pushforward of deformation structures 
along {\em any} isogeny, compatible with Frobenius in a precise way.  

\begin{subsec}{\tt{Pushforward of deformation structures along an isogeny}}
 \label{subsec:pushforward}
 Let $(F, i, \eta)$ be a deformation of $G$ to $R$.  Let $\psi \co F \to F'$ be 
 an isogeny of formal group laws over $R$ of degree $p^r$.  Then $F'$ can be endowed with a 
 deformation structure $(i', \eta')$ such that the following diagram commutes, 
 where $\sigma$ is the absolute $p$-power Frobenius and $\Frob$ is the relative 
 one.  
 \begin{equation}
  \label{pushforward}
  \begin{tikzpicture}[baseline={([yshift=-10pt]current bounding box.north)}, scale=.88]
   \node (AA) at (3, 6) {$F$};
   \node (AB) at (6, 6) {$\pi^* F$};
   \node (AC) at (9, 6) {$i^*G$};
   \node (AD) at (12, 6) {$G$};
   \node (BA) at (1.5, 4) {$F'$};
   \node (BB) at (4.5, 4) {$\pi^* F'$};
   \node (BC) at (8.25, 5) {$i^* G^{(p^r)}$};
   \node (BD) at (11.25, 5) {$G^{(p^r)}$};
   \node (CC) at (7.5, 4) {$i'^* G$};
   \node (CD) at (10.5, 4) {$G$};
   \draw [->] (AB) -- (AA);
   \draw [->] (AB) -- node [above] {$\scriptstyle \eta$} (AC);
   \draw [->] (AC) -- (AD);
   \draw [->] (BB) -- (BA);
   \draw [->] (BB) -- node [above] {$\scriptstyle \hskip .5cm \eta'$} (CC);
   \draw [->] (BC) -- (BD);
   \draw [->] (AA) -- node [left] {$\scriptstyle \psi$} (BA);
   \draw [->] (AB) -- node [left] {$\scriptstyle \pi^* \psi$} (BB);
   \draw [->] (AC) -- node [left] {$\scriptstyle i^* \Frob^r$} (BC);
   \draw [->] (AD) -- node [left] {$\scriptstyle \Frob^r$} (BD);
   \node (AA') at (3, 3) {$\Spf R$};
   \node (AB') at (6, 3) {$\Spf R / \fm$};
   \node (AC') at (9, 3) {$\Spf R / \fm$};
   \node (AD') at (12, 3) {$\Spf k$};
   \node (BA') at (1.5, 1) {$\Spf R$};
   \node (BB') at (4.5, 1) {$\Spf R / \fm$};
   \node (BC') at (8.25, 2) {$\Spf R / \fm$};
   \node (BD') at (11.25, 2) {$\Spf k$};
   \node (CC') at (7.5, 1) {$\Spf R / \fm$};
   \node (CD') at (10.5, 1) {$\Spf k$};
   \draw [->] (AB') -- node [above] {$\scriptstyle \pi$} (AA');
   \draw [double] (AB') -- (AC');
   \draw [->] (AC') -- node [above] {$\scriptstyle i$} (AD');
   \draw [->] (BB') -- node [above] {$\scriptstyle \pi$} (BA');
   \draw [double] (BB') -- (CC');
   \draw [->] (BC') -- node [above] {$\scriptstyle i$} (BD');
   \draw [->] (CC') -- node [above] {$\scriptstyle i'$} (CD');
   \draw [double] (AA') -- (BA');
   \draw [double] (AB') -- (BB');
   \draw [double] (AC') -- (BC');
   \draw [double] (BC') -- (CC');
   \draw [double] (AD') -- (BD');
   \draw [->] (BD') -- node [right] {$\scriptstyle \sigma^r$} (CD');
   \draw [->] (AA) -- (AA');
   \draw [->] (AB) -- (AB');
   \draw [->] (AC) -- (AC');
   \draw [->] (AD) -- (AD');
   \draw [->] (BA) -- (BA');
   \draw [->] (BB) -- (BB');
   \draw [->] (BC) -- (BC');
   \draw [->] (BD) -- (BD');
   \draw [->] (CC) -- (CC');
   \draw [->] (CD) -- (CD');
   \draw [->] (CC) -- (CD);
   \draw [double] (BC) -- (CC);
   \draw [->] (BD) -- (CD);
  \end{tikzpicture}
 \end{equation}
 We write $\psi_! (i, \eta) \ce (i', \eta')$ and call it 
 the {\em pushforward of $(i, 
 \eta)$ along $\psi$}.  Explicitly, the pair is determined by the equalities 
 \[
  i' = i \circ \sigma^r \qquad \ad \qquad \eta' \circ \pi^* \psi = i^* \Frob^r 
  \circ \eta 
 \]
\end{subsec}

\begin{subsec}{\tt{Categories of deformations}}
 Fix $G / k$.  Let $\Def_\isog(R)$ be the category with objects deformations $(F, 
 i, \eta)$ of $G$ to $R$ and with morphisms $(F, i, \eta) \to (F', i', \eta')$, 
 each consisting of an isogeny $\psi \co F \to F'$ of formal group laws over $R$ and 
 an equality $(i', \eta') = \psi_!(i, \eta)$.  The degree of $\psi$ must be $p^r$ 
 for some $r \geq 0$.  Note that the 
 isomorphisms in $\Def_\isog(R)$ are precisely the $\star$-isomorphisms 
 (cf.~\eqref{stariso}, when $r = 0$) and 
 that the only automorphism of an object is the identity by the 
 uniqueness in Proposition \ref{prop:LT}.  
\end{subsec}

\begin{subsec}{\tt{Deformations of Frobenius}}
 \label{subsec:defofrob}
 Given the diagram \eqref{pushforward}, 
 we view a morphism $(F, i, \eta) \to (F', i', \eta')$ in  $\Def_\isog(R)$ as 
 a deformation to $R$ of $\Phi^r$ in the wide category $\FG_\isog$ 
 \eqref{ex:frob}.\footnote{More precisely, with a corresponding 
 wide category of formal group laws understood, it is a deformation of 
 the endomorphism on $G / k$ induced by $\Phi^r$.  We will also denote this 
 endomorphism of formal group laws by $\Phi^r$.}  Thus, we call it a {\em deformation of Frobenius}, and 
 simply call $\psi \co F \to F'$ such if $\eta = \eta'$ so that $\pi^* \psi$ is a relative Frobenius 
 (cf.~\cite[11.3]{cong}).  
 Two deformations $(F_1, i_1, \eta_1) \to (F_1', i_1', \eta_1')$ and 
 $(F_2, i_2, \eta_2) \to (F_2', i_2', \eta_2')$ of Frobenius are {\em isomorphic} 
 if $(F_1, i_1, \eta_1)$ and $(F_2, i_2, \eta_2)$ are $\star$-isomorphic and if 
 $(F_1', i_1', \eta_1')$ and $(F_2', i_2', \eta_2')$ are $\star$-isomorphic.  
\end{subsec}

\begin{prop}[{cf.~\cite[Theorem 4.4]{iph}}]
 \label{prop:S}
 Let $k$, $G$, $R$, $E_n$ be as in Proposition \ref{prop:LT} 
 and again fix $G / k$.  Then, for each $r \geq 0$, the functor 
 \[
  R \mapsto \{\text{isomorphism classes of deformations $(F, i, \eta) \to (F', 
  i', \eta')$ of $\Phi^r$ to $R$}\} 
 \]
 from the category of complete local rings with residue field 
 containing $k$ to the category of sets is co-represented by a ring $A_r$, which 
 is a bimodule over $A_0 = E_n$ with structure maps local homomorphisms $s_r, t_r \co A_0 \to A_r$.  
 Explicitly, there is a (by no means unique) deformation $(F_\univ, \id, \id)$ 
 of $G$ to $A_0$ satisfying the following universal property.  Given any 
 deformation $(F, i, \eta) \to (F', i', \eta')$ of $\Phi^r$ to $R$, there is a 
 unique local homomorphism 
 \[
  \A_r \co A_r \to R 
 \]
 such that $\A_r s_r, \A_r t_r \co A_0 \to R$ reduce to $i, i' \co k = A_0 / I 
 \hookrightarrow R / \fm$ respectively, with $I$ and $\fm$ the maximal ideals, 
 and such that there are unique $\star$-isomorphisms 
 \[
  (F, i, \eta) \to (\A_r^* s_r^* F_\univ, i, \id) \qquad \ad \qquad (F', i', 
  \eta') \to (\A_r^* t_r^* F_\univ, i', \id) 
 \]
\end{prop}
\begin{proof}
 Let $(F_\univ, \id, \id)$ be given from Proposition \ref{prop:LT} and 
 write $\CF_\univ$ for the formal group over $E_n$ whose group law is $F_\univ$ 
 as in \eqref{subsec:fgl}.  Clearly 
 $A_0 = E_n$ with $s_0 = t_0 = \id$.  In general, for each $r \geq 0$, 
 let 
 ${\rm Sub}_r(\CF_\univ)$ be the affine formal scheme over $A_0$ 
 which classifies degree-$p^r$ subgroups of 
 $\CF_\univ$ \cite[Theorem 42]{Str97} 
 and let $A_r$ be its ring of functions.  We need only determine the maps $s_r$, 
 $t_r$ and show that $(F_\univ, \id, \id)$ satisfies the stronger universal 
 property involving $A_r$ as stated.  

 The structure morphism $A_0 \to A_r$ of ${\rm Sub}_r(\CF_\univ) / A_0$ 
 reduces to the identity between residue fields (see [\ib, \S 13]).  
 Thus $F_\univ \times_{A_0} A_r$ inherits the deformation structure 
 $(\id, \id)$ from $F_\univ$ along the base change.  
 Let $\CH^{(p^r)}_\univ \subset \CF_\univ \times_{A_0} A_r$ be the subgroup of 
 degree $p^r$ classified by $\id \co A_r \to A_r$, and let $\CF^{(p^r)}_\univ \ce 
 (\CF_\univ \times_{A_0} A_r) / \CH_\univ$ be the quotient group as in 
 \eqref{subsec:quotient} with a particular group law $F^{(p^r)}_\univ$.  
 The quotient map of formal groups induces an isogeny 
 \[
  \psi^{(p^r)}_\univ \co F_\univ \underset{A_0}{\times} A_r \to F^{(p^r)}_\univ 
 \]
 of group laws over $A_r$.  
 By \eqref{defofrob} it is a deformation 
 of Frobenius \eqref{subsec:defofrob} and we have 
 \[
  {\psi^{(p^r)}_\univ}_!(\id, \id) = (\sigma^r, \id) 
 \]
 In view of Proposition \ref{prop:LT}, let $s_r, t_r \co A_0 \to A_r$ be 
 the unique local homomorphisms which classify $(F_\univ \times_{A_0} A_r, \id, 
 \id)$ and $(F^{(p^r)}_\univ, \sigma^r, \id)$ respectively.  Indeed, by uniqueness, $s_r$ is 
 the structure morphism of ${\rm Sub}_r(\CF_\univ) / A_0$.  

 It remains to verify the universal property.  By Proposition \ref{prop:LT}, 
 given any deformation $\psi \times \psi_! \co (F, i, \eta) \to (F', i', \eta')$ 
 of $\Phi^r$ to $R$, there are unique local homomorphisms 
 \[
  \A, \A' \co A_0 \to R 
 \]
 such that they reduce to $i, i' \co k = A_0 / I \hookrightarrow R / \fm$ 
 respectively, and such that there are unique $\star$-isomorphisms 
 \[
  (F, i, \eta) \to (\A^* F_\univ, i, \id) \qquad \ad \qquad (F', i', \eta') \to 
  (\A'^* F_\univ, i', \id) 
 \]
 Let $\CH \subset \A^* \CF_\univ$ be the image of $\ker\,\psi \subset \CF$ 
 under the first $\star$-isomorphism.\footnote{In fact, since a 
 $\star$-isomorphism of formal group laws can be thought of as a change of 
 coordinates on a formal group, we may write $\CH = \ker\,\psi$.}  It is a subgroup of degree $p^r$.  Then by 
 [\ib, Theorem 42] (taking $X = \Spf A_0$ and $Y = \Spf R$) there is a unique local homomorphism 
 \[
  \A_r \co A_r \to R 
 \]
 which classifies $\CH$ with $\A_r \circ s_r = \A$.  Clearly $\A_r \circ s_r$ reduces to $i$ and there is a 
 unique $\star$-isomorphism $(F, i, \eta) \to (\A_r^* s_r^* F_\univ, i, \id)$ as 
 above.  On the other hand, we have 
 \begin{align*}
  (F', i', \eta') & \stareq (F / \ker\,\psi, i', \eta) & 
                    \text{cf.~\eqref{pushforward}} \\
                  & \stareq (\A^* F_\univ / H, i', \id) \\
                  & = \big( \A_r^* (F_\univ \times_{A_0} A_r) / \A_r^* 
                    H^{(p^r)}_\univ, i', \id \big) \\
                  & = (\A_r^* F^{(p^r)}_\univ, i', \id) & \text{by [\ib, Theorem 
                    19\,(v)]} \\
                  & \stareq (\A_r^* t_r^* F_\univ, i', \id) 
 \end{align*}
 Therefore, $\A_r \circ t_r = \A'$, so $\A_r \circ t_r$ reduces to $i'$ and 
 there is a unique $\star$-isomorphism $(F', i', \eta') \to (\A_r^* t_r^* 
 F_\univ, i', \id)$.  
\end{proof}

\begin{subsec}{\tt{Canonical lifts of Frobenius morphisms}}
 \label{subsec:canlift}
 To summarize Proposition \ref{prop:S} and its proof, 
 the ring $A_r$ carries a universal example $\psi^{(p^r)}_\univ$ of deformation of 
 $\Phi^r$ to $R$ as follows.\footnote{See [\ib, \S 10, \S 13] for more about the rings $A_r$.  
 For an explicit example, see \cite[Theorem 1.2]{me} where $r = 1$ and $G$ is of 
 height $2$ over $k = \cF_p$.}  
 \begin{equation}
  \label{univex}
  \begin{tikzpicture}[baseline={([yshift=-10pt]current bounding box.north)}, scale=.88]
   \node (AA) at (3.9, 6) {$s_r^* F_\univ =$};
   \node (AB) at (6, 5.96) {$F_\univ \times_{A_0} A_r$};
   \node (AC) at (9, 5.96) {\raisebox{.25cm}{$F^{(p^r)}_\univ$}};
   \node (AD) at (10.4, 6.05) {$\stareq t_r^* F_\univ$};
   \draw [->] (AB) -- node [above] {$\scriptstyle \psi^{(p^r)}_\univ$} (AC);
   \node (A') at (7.5, 3.7) {$\Spf A_r$};
   \draw [->] (AB) -- (A');
   \draw [->] (AC) -- (A');
  \end{tikzpicture}
 \end{equation}
 The central notion of norm coherence in this paper, introduced in the next section, concerns the question of when the $\star$-isomorphism 
 in the above diagram is the identity.  We write $\psi^{(p^r)\,\star}_\univ$ 
 for the composite of $\psi^{(p^r)}_\univ$ with this $\star$-isomorphism.  
\end{subsec}

\begin{rmk}
 \label{rmk:unique}
 Continuing with \eqref{subsec:unique}, 
 we see from the proof of Proposition \ref{prop:S} that the maps $\A_r$, $s_r$, 
 $t_r$ are independent of 
 the choice between {\em $\star$-isomorphic} universal deformations.  
\end{rmk}

\begin{subsec}{\tt{Dependency of $F_\univ$ on $G / k$}}
 \label{subsec:univ}
 The choice of $F_\univ = F_\univ(G)$ as in (\ref{subsec:unique}, 
 \ref{rmk:unique}) 
 can be made functorial with respect to morphisms in $\FG_\isog$.  Specifically, 
 for functoriality under base change, the right square in 
 \eqref{widemor} deforms so that
 \[
  F_\univ(G' \times_{k'} k) = F_\univ(G') \underset{E_n(G')}{\times} 
  E_n(G' \times_{k'} k) 
 \]
 as formal group laws, 
 where $n$ is the height of $G' / k'$ (invariant under base change) and 
 $E_n(G') \to E_n(G' \times_{k'} k)$ sends each generator of 
 the source to the corresponding one of the target.  
 This identity follows from the construction of $F_\univ$ in \cite[Proposition 1.1]{LT} 
 as a ``generic group law'' $\G$.  
 
 Moving to the left square of \eqref{widemor}, 
 let $\psi \co \CG \to \CG' \times_{k'} k$ be any isogeny 
 of degree $p^r$ and consider the following (cf.~\eqref{factor} and \eqref{pushforward}).  
 \[
  \begin{tikzpicture}[scale=.88]
   \node (LT) at (0, 3) {$\CG$};
   \node (MLT) at (4, 3) {$\CG' \times_{k'} k$};
   \node (MRT) at (8, 3) {$\CG^{(p^r)}$};
   \node (RT) at (12, 3) {$\CG$};
   \node (LB) at (0, 0) {$\Spf k$};
   \node (MLB) at (4, 0) {$\Spf k$};
   \node (MRB) at (8, 0) {$\Spf k$};
   \node (RB) at (12, 0) {$\Spf k$};
   \draw [->] (LT) -- node [above] {$\scriptstyle \psi$} (MLT);
   \draw [->] (MLT) -- node [above] {$\scriptstyle \sim$} (MRT);
   \draw [->] (MRT) -- (RT);
   \draw [->] (LT) -- (LB);
   \draw [->] (MLT) -- (MLB);
   \draw [->] (MRT) -- (MRB);
   \draw [->] (RT) -- (RB);
   \draw [double] (LB) -- (MLB);
   \draw [double] (MLB) -- (MRB);
   \draw [->] (MRB) -- node [above] {$\scriptstyle \sigma^r$} (RB);
   \draw [->] (LT) to [out = 30, in = 150] node [above] {$\scriptstyle \Frob^r$} (MRT);
   \node at (8.5, 2.5) {$\lrcorner$}; 
  \end{tikzpicture}
 \]
 Note that $\Frob^r$ deforms to 
 $\psi_\univ^{(p^r)\,\star} \co s_r^* F_\univ(G) \to t_r^* F_\univ(G)$ over $A_r$ as in \eqref{univex}.  
 Moreover, 
 \begin{align*}
  t_r^* F_\univ(G) & = F_\univ(G) \underset{A_0}{\times}^{t_r} A_r \\
                   & = F_\univ(G) \underset{E_n(G)}{\times} E_n \big( G^{(p^r)} 
                     \big) \underset{E_n \left( G^{(p^r)} \right)}{\times} A_r 
                     \\
                   & = F_\univ \big( G^{(p^r)} \big) \underset{E_n \left( 
                     G^{(p^r)} \right)}{\times} A_r 
 \end{align*}
 where $E_n \big( G^{(p^r)} \big) = \BW k \lb v_1, \ldots, v_{n - 1} \rb$ 
 with each $v_i \mapsto t_r(u_i) \in A_r$.\footnote{For an explicit 
 example, see \cite[Theorem 1.6\,(ii)]{me} where $r = 1$ and $G$ is of 
 height $n = 2$ over $k = \cF_p$, with $u_1 = h$, $t_1(u_1) = \psi^{\,p}(h)$.}  
 The construction of $F_\univ$ clearly respects isomorphisms so that 
 \[
  F_\univ(G' \times_{k'} k) 
  \underset{E_n(G' \times_{k'} k)}{\times} 
  E_n \big( G^{(p^r)} \big) 
  \cong F_\univ \big( G^{(p^r)} \big) 
 \]
 Thus, over $A_r$ (omitting the base changes), $\psi$ deforms to 
 \[
  F_\univ(G) \xrightarrow{\psi_\univ^{(p^r)\star}} F_\univ \big( G^{(p^r)} \big) 
  \xrightarrow{\sim} F_\univ(G' \times_{k'} k) 
 \]
 This shows the functoriality of $F_\univ$ under isogenies.  

 To summarize, given a morphism in $\FG_\isog$ as above, 
 the universal deformations of its source and target can be chosen 
 so that in terms of formal group laws \eqref{widemor} deforms over $A_r$ 
 as follows.  
 \[
  \begin{tikzpicture}[scale=.88]
   \node (AA) at (0, 2) {$s_r^* F_\univ(G)$};
   \node (AB) at (4, 2) {$t_r^* F_\univ(G)$};
   \node (BB) at (6, 4) {$F_\univ \big( G^{(p^r)} \big)$};
   \node (CA) at (4, 6) {$F_\univ(G)$};
   \node (CB) at (8, 6) {$F_\univ(G' \times_{k'} k)$};
   \node (CC) at (12, 6) {$F_\univ(G')$};
   \draw [->] (AA) -- (AB);
   \draw [->] (AA) -- (CA);
   \draw [->] (AB) -- (BB);
   \draw [->] (BB) -- (CB);
   \node (AA') at (0, -1) {$\Spf A_r$};
   \node (AB') at (4, -1) {$\Spf A_r$};
   \node (BB') at (6, 1) {$\Spf E_n \big( G^{(p^r)} \big)$};
   \node (CA') at (4, 3) {$\Spf E_n(G)$};
   \node (CB') at (8, 3) {$\Spf E_n(G' \times_{k'} k)$};
   \node (CC') at (12, 3) {$\Spf E_n(G')$};
   \draw [double] (AA') -- (AB');
   \draw [->] (CB') -- (CC');
   \draw [->] (AA') -- (CA');
   \draw [->] (AB') -- (BB');
   \draw [->] (BB') -- (CB');
   \draw [->] (AA) -- (AA');
   \draw [->] (AB) -- (AB');
   \draw [->] (BB) -- (BB');
   \draw [->] (CA) -- (CA');
   \draw [->] (CB) -- (CB');
   \draw [->] (CC) -- (CC');
   \draw [->] (CB) -- (CC);
   \node at (8.5, 5.5) {$\lrcorner$}; 
   \node at (0.5, 1.5) {$\mathbin{\rotatebox[origin=c]{-20}{$\scriptstyle \rangle$}}$}; 
   \node at (4.5, 1.5) {$\mathbin{\rotatebox[origin=c]{-20}{$\scriptstyle \rangle$}}$}; 
  \end{tikzpicture}
 \]
\end{subsec}

\section{Norm-coherent deformations}
\label{sec:nc}

\begin{subsec}{\tt{Set-up}}
 \label{subsec:setup'}
 Let $k$ be an algebraic extension of $\BF_p$ (in particular, $k$ is perfect) 
 and $G$ be a formal group law over 
 $k$ of finite height $n$.  Let $R$ be a complete local ring with 
 maximal ideal $\fm$ and residue field $R / \fm \supset k$.  Let $F / R$ be a 
 deformation of $G / k$ with deformation structure $(i, \id)$ as in \eqref{subsec:ds}.  
 \begin{rmk}
  \label{rmk:normalization}
  Observe that, given any deformation $(F, i, \eta)$, 
  there exists a unique deformation $(\tF, i, \id)$ 
  such that the two are in the same $\star$-isomorphism class 
  (cf.~\eqref{refine}).  Without loss of generality, here we 
  focus on the case of $\eta = \id$.  
 \end{rmk}
\end{subsec}

\begin{subsec}{\tt{Quotient by the $p$-torsion subgroup}}
 \label{subsec:f_p}
 As in \eqref{subsec:fgl} write $\CF$ for the formal group over $R$ whose group law is $F$ 
 (upon choosing a coordinate) and write 
 $\CF[p]$ for its subgroup scheme of $p$-torsions.  This is defined over an extension 
 $\tR$ of $R$ obtained by adjoining the roots of the $p$-series of $F$.  
 Let $\CF / \CF[p] \ce (\CF \times_R \tR) / \CF[p]$ be the quotient 
 group as in \eqref{subsec:quotient} with a particular group law $F / 
 F[p]$ so that the isogeny 
 \[
  f_p \co F \to F / F[p] 
 \]
 induced by the quotient morphism of formal groups 
 is a deformation of Frobenius \eqref{subsec:defofrob}.  
 Note that $\CF[p](\tR)$ is 
 stable under the action of $\Aut(\tR / R)$.  Thus $\,f_p$ can be defined 
 over $R$ (cf.~\cite[Theorem 1.4]{Lubin67}).  

 \begin{rmk}
  \label{rmk:endo}
  The restriction of $\,f_p$ on the special fiber is the relative $p^n$-power 
  Frobenius.  It is not an endomorphism unless $k \subset \BF_{p^n}$ 
  (cf.~\cite[proof of Proposition 2.5.1]{Ando95}).  
 \end{rmk}
\end{subsec}

\begin{subsec}{\tt{The isogeny $l_p$}}
 By Proposition \ref{prop:S} there is a unique local homomorphism 
 $\A_n \co A_n \to R$ 
 together with a unique $\star$-isomorphism 
 $(F/F[p], i \circ \sigma^n, \id) \to (\A_n^* t_n^* F_\univ, i \circ \sigma^n, \id)$.  
 Write 
 \[
  g_p \co F / F[p] \to \A_n^* t_n^* F_\univ 
 \]
 for the corresponding 
 $\star$-isomorphism of formal group laws.  Let 
 \[
  l_p \co F \to \A_n^* t_n^* F_\univ 
 \]
 be the composite $g_p \circ f_p$.  
 
 \begin{rmk}
  \label{rmk:l_p}
  The isogeny $l_p$ of formal group laws over $R$ is uniquely characterized by the 
  following properties (cf.~[\ib, Proposition 2.5.4], the proof here being completely analogous).  
  \begin{enumerate}[(i)]
   \item It has source $F$ and target 
   of the form $\A^* t_n^* F_\univ$ for some local homomorphism $\A \co A_n \to 
   R$.  

   \item The kernel of $l_p$ applied to $\CF$ is $\CF[p]$.  

   \item Over the residue field, $l_p$ reduces to the relative $p^n$-power 
   Frobenius.  
  \end{enumerate}
  Explicitly, with notation as in \eqref{pushforward}, $f_p$ and $l_p$ 
  fit into the following commutative diagram.  Their restrictions on
  the special fiber are 
  highlighted with corresponding decorations, which are in fact identical.  
  \begin{equation}
   \label{l_p}
   \begin{tikzpicture}[baseline={([yshift=-10pt]current bounding box.north)}, scale=.88]
    \node (AA) at (3, 4) {$F$};
    \node (AB) at (6, 4) {$\pi^* F$};
    \node (AC) at (9, 4) {$i^* G$};
    \node (AD) at (12, 4) {$G$};
    \node (BA) at (1.5, 2) {$F / F[p]$};
    \node (BB) at (4.5, 2) {$\pi^* (F / F[p])$};
    \node (BC) at (7.5, 2) {$i^* G^{(p^n)}$};
    \node (BD) at (10.5, 2) {$G^{(p^n)}$};
    \node (CA) at (0, 0) {$\A_n^* t_n^* F_\univ$};
    \node (CB) at (3, 0) {$\pi^* \A_n^* t_n^* F_\univ$};
    \node (CC) at (6, 0) {$i'^* G$};
    \node (CD) at (9, 0) {$G$};
    \draw [->] (AB) -- (AA);
    \draw [double] (AB) -- (AC);
    \draw [->] (AC) -- (AD);
    \draw [->] (BB) -- (BA);
    \draw [double] (BB) -- (BC);
    \draw [->] (BC) -- (BD);
    \draw [->] (CB) -- (CA);
    \draw [double] (CB) -- (CC);
    \draw [->] (CC) -- (CD);
    \draw [->, line join = round, decorate, decoration = {zigzag, segment length 
           = 4, amplitude = .9, post = lineto, post length = 2pt}] (AA) -- node 
          [left] {$\scriptstyle f_p$} (BA);
    \draw [->, line join = round, decorate, decoration = {zigzag, segment length 
           = 4, amplitude = .9, post = lineto, post length = 2pt}] (AB) -- node 
          [left] {$\scriptstyle \pi^* f_p$} (BB);
    \draw [->, line join = round, decorate, decoration = {zigzag, segment length 
           = 8, amplitude = .9, post = lineto, post length = 2pt}] (AC) -- node 
          [left] {$\scriptstyle i^* \Frob^n$} (BC);
    \draw [->] (AD) -- node [left] {$\scriptstyle \Frob^n$} (BD);
    \draw [->] (BA) -- node [left] {$\scriptstyle g_p$} (CA);
    \draw [double] (BB) -- node [left] {$\scriptstyle \pi^* g_p$} (CB);
    \draw [double] (BC) -- (CC);
    \draw [->] (BD) -- (CD);
    \draw [->, line join = round, decorate, decoration = {zigzag, segment length 
           = 8, amplitude = .9, post = lineto, post length = 2pt}] (AA) to [out = 
          -170, in = 80] node [left] {$\scriptstyle l_p$} (CA);
   \end{tikzpicture}
  \end{equation}
 \end{rmk}
 
 \begin{ex}
  \label{ex:Honda}
  Let $k = \BF_p$ and $G$ be the Honda formal group law given by [\ib, 2.5.5] 
  so that $[p]_{_G}(t) = t^{\,p^n}$.  Then the relative Frobenius 
  $\Frob^n$ coincides with the absolute Frobenius automorphism on $G$ and so $l_p 
  = [p]_{_F}$ (cf.~[\ib, Proposition 2.6.1]).  
 \end{ex}
\end{subsec}

\begin{subsec}{\tt{The isogenies $l_{_H}$}}
 \label{subsec:l_H}
 More generally, let $\CH \subset \CF$ be a subgroup of degree $p^r$, $\psi_{_H} \co 
 F \to F'$ be any isogeny with kernel $\CH$, and 
 $\psi_{_H} \times {\psi_{_H}}_! \co (F, i, \id) \to (F', i', \eta')$ 
 be the corresponding deformation of Frobenius.  
 The diagram \eqref{l_p} generalizes as follows.  
 \begin{equation}
  \label{l_H}
  \begin{tikzpicture}[baseline={([yshift=-10pt]current bounding box.north)}, scale=.88]
   \node (AA) at (3, 4) {$F$};
   \node (AB) at (6, 4) {$\pi^* F$};
   \node (AC) at (9, 4) {$i^* G$};
   \node (AD) at (12, 4) {$G$};
   \node (BA) at (1.5, 2) {$F'$};
   \node (BB) at (4.5, 2) {$\pi^* F'$};
   \node (BC) at (7.5, 2) {$i^* G^{(p^r)}$};
   \node (BD) at (10.5, 2) {$G^{(p^r)}$};
   \node (CA) at (0, 0) {$\A_r^* t_r^* F_\univ$};
   \node (CB) at (3, 0) {$\pi^* \A_r^* t_r^* F_\univ$};
   \node (CC) at (6, 0) {$i'^* G$};
   \node (CD) at (9, 0) {$G$};
   \draw [->] (AB) -- (AA);
   \draw [double] (AB) -- node [above] {$\scriptstyle \eta\,=\,\id$} (AC);
   \draw [->] (AC) -- (AD);
   \draw [->] (BB) -- (BA);
   \draw [->] (BB) -- node [above] {$\scriptstyle \eta'$} (BC);
   \draw [->] (BC) -- (BD);
   \draw [->] (CB) -- (CA);
   \draw [double] (CB) -- (CC);
   \draw [->] (CC) -- (CD);
   \draw [->, line join = round, decorate, decoration = {zigzag, segment length 
          = 4, amplitude = .9, post = lineto, post length = 2pt}] (AA) -- node 
         [left] {$\scriptstyle \psi_{_H}$} (BA);
   \draw [->, line join = round, decorate, decoration = {zigzag, segment length 
          = 4, amplitude = .9, post = lineto, post length = 2pt}] (AB) -- node 
         [left] {$\scriptstyle \pi^* \psi_{_H}$} (BB);
   \draw [->, line join = round, decorate, decoration = {zigzag, segment length 
          = 8, amplitude = .9, post = lineto, post length = 2pt}] (AC) -- node 
         [left] {$\scriptstyle i^* \Frob^r$} (BC);
   \draw [->] (AD) -- node [left] {$\scriptstyle \Frob^r$} (BD);
   \draw [->] (BA) -- (CA);
   \draw [->] (BB) -- (CB);
   \draw [double] (BC) -- (CC);
   \draw [->] (BD) -- (CD);
   \draw [->, line join = round, decorate, decoration = {zigzag, segment length 
          = 8, amplitude = .9, post = lineto, post length = 2pt}] (AA) to [out = 
         -170, in = 80] node [left] {$\scriptstyle l_{_H}$} (CA);
  \end{tikzpicture}
 \end{equation}
 In particular, when $\psi_{_H} = f_{_H} \co F \to F / H$ is the 
 deformation of Frobenius with kernel $\CH$, 
 we have the following commutative diagram.  
 \begin{equation}
  \label{fgl}
  \begin{tikzpicture}[baseline={([yshift=-10pt]current bounding box.north)}, scale=.88]
   \node (LT) at (0, 3) {$F$};
   \node (RT) at (4, 3) {$F / H$};
   \node (RB) at (4, 0) {$\A_r^* t_r^* F_\univ$};
   \draw [->] (LT) -- node [above] {$\scriptstyle f_{_H}$} (RT);
   \draw [->] (LT) -- node [below] {$\scriptstyle l_{_H}$} (RB);
   \draw [->] (RT) -- node [right] {$\scriptstyle g_{_H}$} (RB);
  \end{tikzpicture}
 \end{equation}

 \begin{rmk}
  This construction of $l_{_H}$ is functorial under base change and 
  under quotient, due to the functoriality of $f_{_H}$ and $g_{_H}$ 
  (see \cite[Theorem 19\,(v)]{Str97}, 
  \cite[Proposition 2.2.6]{Ando95}, 
  and Proposition \ref{prop:S}).  
  To be precise, given any local homomorphism 
  $\B \co R \to R'$ and any finite subgroups 
  $\CH_1 \subset \CH_2$ of $\CF$, we have 
  \[
   l_{_{\B^* H}} = \B^* l_{_H} \qquad \ad \qquad 
   l_{_{H_2 / H_1}} \circ l_{_{H_1}} = l_{_{H_2}} 
  \]
  where the composition is taken up to a $\star$-isomorphism, 
  as shown in the following commutative diagrams.  
  \begin{equation}
   \label{bc}
   \begin{tikzpicture}[baseline={([yshift=-10pt]current bounding box.north)}, scale=.88]
    \node (LT) at (-2.1, 4) {$F$};
    \node (MT) at (2.5, 4) {$\B^* F$};
    \node (LM) at (-2.1, 2) {$F / H$};
    \node (MM) at (2.5, 2) {$\B^*(F / H)$};
    \node (RM) at (4.58, 2) {$= \B^* F / \B^* H$};
    \node (LB) at (-2.1, 0) {$\A_r^* t_r^* F_\univ$};
    \node (MB) at (2.5, 0) {$\B^* \A_r^* t_r^* F_\univ$};
    \draw [->] (MT) -- (LT);
    \draw [->] (MM) -- (LM);
    \draw [->] (MB) -- (LB);
    \draw [->] (LT) -- node [left] {$\scriptstyle f_{_H}$} (LM);
    \draw [->] (LM) -- node [left] {$\scriptstyle g_{_H}$} (LB);
    \draw [->] (MT) -- node [left] {$\scriptstyle \B^* f_{_H}$} (MM);
    \draw [->] (MM) -- node [left] {$\scriptstyle \B^* g_{_H}$} (MB);
    \draw [->] (MT) -- node [right] {$\scriptstyle f_{_{\B^* H}}$} (RM);
    \draw [->] (RM) -- node [right] {$\scriptstyle g_{_{\B^* H}}$} (MB);
   \end{tikzpicture}
  \end{equation}
  \begin{equation}
   \label{pf}
   \begin{tikzpicture}[baseline={([yshift=-10pt]current bounding box.north)}, scale=.88]
    \node (LT) at (-.5, 4) {$F$};
    \node (MT) at (3.5, 4) {$F / H_2$};
    \node (RT) at (6.5, 4) {$\A_{r_2}^* t_{r_2}^* F_\univ$};
    \node (LM) at (-.5, 2) {$F / H_1$};
    \node (MM) at (3.5, 2) {$F / H_1 \!\Big/\! H_2 / H_1$};
    \node (LB) at (-.5, 0) {$\A_{r_1}^* t_{r_1}^* F_\univ$};
    \node (MB) at (3.5, 0) {$\A_{r_1}^* t_{r_1}^* F_\univ \big/ H^{\,(p^{^{r_2 - r_1}}\!)}_\univ$};
    \draw [->] (LT) -- node [above] {$\scriptstyle f_{_{H_2}}$} (MT);
    \draw [->] (MT) -- node [above] {$\scriptstyle g_{_{H_2}}$} (RT);
    \draw [->] (LM) -- node [above] {$\scriptstyle f_{_{H_2 / H_1}}$} (MM);
    \draw [->] (LB) -- node [above] {$\scriptstyle \tilde{f}_{_{H_2 / H_1}}$} (MB);
    \draw [->] (LT) -- node [left] {$\scriptstyle f_{_{H_1}}$} (LM);
    \draw [->] (LM) -- node [left] {$\scriptstyle g_{_{H_1}}$} (LB);
    \draw [double] (MT) -- (MM);
    \draw [->] (MM) -- node [left] {$\scriptstyle \tilde{g}_{_{H_1}}$} (MB);
    \draw [->] (MB) -- node [right] {$\scriptstyle \tilde{g}_{_{H_2 / H_1}}$} (RT);
   \end{tikzpicture}
  \end{equation}
 \end{rmk}
\end{subsec}

\begin{subsec}{\tt{Definition of norm coherence}}
 \label{subsec:ncdef}
 Recall from the proof of Lemma \ref{lem:norm} that 
 $\Norm_{\psi^*} \co \CO_{G}$ $\to \CO_{G'}$ sends 
 a coordinate $x_{_G}$ on $\CG$ to the coordinate on $\CG' = \CG^{(p^r)}$ which pulls 
 back along $\psi = \Frob^r \co \CG \to \CG^{(p^r)}$ to $x_{_G}^{\,p^r}$.  In other words, 
 the norm map agrees with pushing forward a coordinate along the Frobenius isogeny.  

 This agreement on $x_{_G}$ over $k$ may not extend to $R$ for an arbitrary 
 coordinate $x$ on $\CF$ lifting $x_{_G}$.  On one hand, 
 given a subgroup $\CH \subset \CF$ of degree $p^r$, 
 the isogeny $f_{_H} \co \CF \to \CF / \CH$ lifts the norm map 
 in the sense that 
 \begin{align*}
  \label{x_H}
  x_{_H} \big(\,f_{_H}(P) \big) & = \prod_{Q \in \CH(\tR)} \big( x(P) 
                                  \underset{F}{+} x(Q) \big) & \text{by 
                                  \eqref{lubinisog}} \st \\
                                & = \prod_{\sigma \in \Aut(\CO_F / \CO_{F / H})} 
                                  \sigma \cdot x\,(P) \\
                                & = ~f_{_H}^* \, \Norm_{\,f_{_H}^*}(x) \, (P) \\
                                & = \Norm_{\,f_{_H}^*}(x) \big(\,f_{_H}(P) \big) 
                                  & \text{cf.~\eqref{norm}} 
 \end{align*}
 where $x_{_H}$ is the coordinate corresponding to the group law $F_x / H$, 
 $P$ is any $R$-point on $\CF$, 
 and $\tR$ is an extension of $R$ 
 to define the $p^r$ points of $\CH$.\footnote{A 
 detailed proof for the third equality in \eqref{x_H} for the 
 norm map, in the context of Galois theory analogous to the situation 
 here, can be found in \cite[pp.~916-920, esp.~Corollary 10.87]{AMA}.  
 Moreover, consider a coordinate on $\CF$ as a map $\CF \to \hat{\BA}^1$ 
 \eqref{subsec:fgl}.  
 We then have 
 \[
  \CF \xrightarrow{~f_{_H}} \CF' \to \hat{\BA}^1 
 \]
 and $\Norm_{\,f_{_H}^*}$ gives 
 \[
  \CF \xrightarrow{x} \hat{\BA}^1 \quad \mapsto \quad \CF' \xrightarrow{~x_{_H}} \hat{\BA}^1 
 \]
 which is analogous to a {\em norm map} as a piece of structure in a 
 Tambara functor \cite[3.1]{Tambara}.  
 This last notion of a norm map has been packaged into 
 equivariant stable homotopy theory and turned out as a key ingredient 
 in recent advances in the field \cite{Brun, Hill-Hopkins}.}  
 On the other hand, the isogeny $l_{_H} = g_{_H} \circ \, f_{_H}$ 
 lifts $\Frob^r$ canonically with respect to $\CH$; that is, if $\,f_{_H}'$ is another lift with kernel 
 $\CH$ and classifying $\star$-isomorphism $g_{_H}'$, then $g_{_H}' \circ f_{_H}' = l_{_H}$ 
 (Remark \ref{rmk:l_p}).  

 \begin{defn}
  \label{def:nc}
  Let $(F, i, \id)$ be a deformation of $G$ to $R$ as in \eqref{subsec:setup'}.  
  We say that it is {\em norm-coherent} 
  if given any finite subgroup $\CH$ of $\CF$, the identity 
  \begin{equation}
   \label{l=f}
   l_{_H} = \, f_{_H} 
  \end{equation}
  holds.  In other words, the condition is 
  that the $\star$-isomorphism $g_{_H} = \id$.  

  More generally, given any deformation $(F, i, \eta)$ of $G$ to $R$, 
  let $(\tF, i, \id)$ be the unique deformation associated to it (Remark 
  \ref{rmk:normalization}).  
  We say that $(F, i, \eta)$ is norm-coherent 
  if $(\tF, i, \id)$ is.  

  With the deformation structure understood, 
  we also call the formal group law $F$, 
  as well as its corresponding coordinate $x_{_F}$ on $\CF$, norm-coherent.  
 \end{defn}

 \begin{rmk}
  \label{rmk:nc}
  Explicitly, in terms of a norm-coherent coordinate $x$, 
  \eqref{l=f} boils down to the identity 
  \[
   h^{(p^r)} \big( l_{_H}(x) \big) = 
   \prod_{c \in \CH(\tR)} h \big( x(c) \underset{F}{+} x \big) 
  \]
  for all $h(x) \in R \lb x \rb$, where $h^{(p^r)}$ is the series obtained by 
  twisting the coefficients of $h$ with 
  the automorphism on $R$ which lifts the absolute $p^r$-power Frobenius (cf.~\eqref{norm'}).  
  A more conceptual form of this condition is 
  \begin{equation}
   \label{norm''}
   x_{_{F'\!, \psi_!(i, \eta)}} = \Norm_{\psi^*}(x_{_{F, i, \eta}}) 
  \end{equation}
  for any deformation $\psi \times \psi_!$ of Frobenius (cf.~\eqref{x_H}).  
  Indeed, if the isogeny 
  $\psi$ has kernel $\CH$, the pushforward $\psi_!(i, \eta)$ of deformation 
  structure indicates a change of coordinates on the target so that the left-hand side of 
  \eqref{norm''} corresponds to the formal group law $\A_r^* t_r^* F_\univ$ as in \eqref{l_H}.  
 \end{rmk}
\end{subsec}

\begin{subsec}{\tt{Functoriality of norm coherence}}
 \label{subsec:fun}
 Recall from \eqref{subsec:bc} and \eqref{subsec:pushforward} 
 the operations of base change and pushforward of deformation structures.  
 The notion of norm coherence in Definition \ref{def:nc} 
 is preserved under both as follows.  
 
 \begin{prop}
  \label{prop:fun}
  Let $(F, i, \eta)$ be a norm-coherent deformation of $G$ to $R$.  
  \begin{enumerate}[(i)]
   \item Given any local homomorphism $\B \co R \to R'$, the 
   deformation $\big( \B^* F, \B^*(i, \eta) \big)$ is norm-coherent.  

   \item Given any isogeny $\psi \co F \to F'$ over $R$, the deformation 
   $\big( F', \psi_!(i, \eta) \big)$ is norm-coherent.  
   In particular, given any 
   finite subgroup $\CH \subset \CF$ of degree $p^r$, 
   the deformation $(F / H, i \circ \sigma^r, \eta)$ is norm-coherent.  
  \end{enumerate}
 \end{prop}
 \begin{proof}
  For (i), first note that 
  \begin{align*}
            & ~ (\tF, i, \id) \stareq (F, i, \eta) \\
   \implies & ~ \big( \B^* \tF, \B^*(i, \id) \big) \stareq 
              \big( \B^* F, \B^*(i, \eta) \big) \\
   \implies & ~ \B^* \tF = \widetilde{\B^* F~} 
  \end{align*}
  To see that $\B^* \tF$ is norm-coherent, we have from \eqref{bc} 
  \[
   l_{_{\B^* H}} = \B^* l_{_H} = \B^* f_{_H} = f_{_{\B^* H}} 
  \]

  For (ii), suppose that $\psi$ is of degree $p^r$ 
  and let $i' = i \circ \sigma^r$.  
  In view of 
  \[
   \big( F', \psi_!(i, \eta) \big) 
   \stareq (F / H, i', \eta) 
   \stareq (\widetilde{F / H \,}\!, i', \id) 
   = (\tF / H, i', \id) 
  \]
  we are reduced to the special case of 
  \[
   (F, i, \id) \xrightarrow{\,f_{_H} \times \, {f_{_H}}_!} (F / H, i \circ \sigma^r, \id) 
  \]
  Since the source is norm-coherent, we have from \eqref{pf} 
  \[
   l_{_H} = f_{_H} \quad \ad \quad 
   l_{_{K / H}} \circ l_{_H} = l_{_K} = f_{_K} = f_{_{K / H}} \circ f_{_H} 
  \]
  where $\CK$ is any finite subgroup of $\CF$ containing $\CH$, 
  and the first composition is on-the-nose 
  because of the first identity in the display.  
  Given that $g_{_{K / H}}$ is an isomorphism, we then deduce from these 
  \[
   l_{_{K / H}} = \, f_{_{K / H}} 
  \]
  which shows the norm coherence of $(F / H, i \circ \sigma^r, \id)$.  
 \end{proof}
\end{subsec}

\section{Existence and uniqueness of norm-coherent deformations}
\label{sec:nc'}

The following generalizes a result of Ando's.  

\begin{prop}[{cf.~\cite[Theorem 2.5.7]{Ando95}}]
 \label{prop:nc}
 Let $k$, $G$, $R$, $F$ be as in \eqref{subsec:setup'} and fix $G / k$.  There 
 exists a unique formal group law $F'$ over $R$, $\star$-isomorphic to $F$, 
 that is norm-coherent.  In other words, given any coordinate $x_{_G}$ on the 
 formal group $\CG$ and a coordinate $x_{_F}$ on $\CF$ that lifts $x_{_G}$, 
 there exists a unique norm-coherent coordinate on $\CF$ whose corresponding formal 
 group law is $\star$-isomorphic to that of $x_{_F}$.  
\end{prop}

To show this, we will follow Ando's proof of his theorem, making 
alterations for greater generality whenever necessary 
(most significantly in \eqref{ddef}).  The argument breaks into 
two parts, the first focusing on norm coherence for the $p$-torsion subgroup 
$\CF[p]$ and the second showing functoriality for all finite subgroups.  
We begin with the following key lemma.  

\begin{lem}[{cf.~[\ib, Theorem 2.6.4]}]
 \label{klem}
 Given any coordinate $x_{_F}$ on $\CF$ that lifts $x_{_G}$, 
 there exists a unique coordinate on $\CF$ whose corresponding formal 
 group law is $\star$-isomorphic to that of $x_{_F}$ and satisfies 
 \begin{equation}
  \label{nc}
  l_p = \, f_p 
 \end{equation}
\end{lem}
\begin{proof}
 \hskip .36cm {\bf Existence} \hskip .18cm First we reduce the proof to the universal 
 case.  Let $F_\univ$ be a universal deformation of $G / k$ to $E_n$ as in 
 Proposition \ref{prop:LT}, so that there is a unique local homomorphism 
 \[
  \A \co E_n \to R 
 \]
 together with a unique $\star$-isomorphism 
 \[
  g \co F \to \A^* F_\univ 
 \]
 Suppose that we can construct a coordinate $x$ on $\CF_\univ$ 
 whose corresponding formal group law $F_\univ'$ satisfies \eqref{nc} and is 
 $\star$-isomorphic to $F_\univ$.  
 Taking $\CH = \CF_\univ [p]$ in the proof of Proposition \ref{prop:fun}\,(i), 
 we then see that $\A^* F_\univ'$ satisfies \eqref{nc} and 
 is $\star$-isomorphic to $F$.  

 We turn to the universal case.  The proof is inductive, on powers of the 
 maximal ideal $I$ of $E_n$.  Let $y$ be the coordinate 
 corresponding to $F_\univ$ from above, so we may write $F_y \ce F_\univ$.   
 With respect to $y$, given that 
 $g_p^y \co F_y / F[p] \to \A_n^* t_n^* F_y$ is defined over $E_n$ as in 
 \eqref{subsec:f_p}, let $a(t) \in E_n \lb t \rb$ be such that 
 \begin{equation}
  \label{adef}
  g_p^y(t) = t + a(t) 
 \end{equation}
 We shall construct a desired coordinate $x$ on the universal formal group $\CF$ by inductively modifying the 
 coordinate $y$ so that $a(t) \equiv 0 \md I^r$ for increasing $r$.  
 
 Let the inductive hypothesis be 
 \begin{equation}
  \label{ahyp}
  a(t) = \sum_{j \geq 1} a_j \, t^{\,j} \hskip 1cm \text{with ~ $a_j \in 
  I^{r - 1}$} 
 \end{equation}
 Since $g_p^y$ is a $\star$-isomorphism, we get automatically the case $r = 2$.  
 Let $\d(t)$ be the power series 
 \begin{equation}
  \label{ddef}
  \d(t) = t - a^{(-p^n)}(t) 
 \end{equation}
 where $a^{(-p^n)}(t)$ is the series obtained by twisting the coefficients 
 $a_j$ with the inverse of the local automorphism $\A_n t_n$ on $A_0 = E_n$ 
 and has its coefficients in $I^{r - 1}$ as well.\footnote{By Proposition \ref{prop:S}, 
 $\A_n t_n$ lifts the $p^n$-power Frobenius $i' = \sigma^n$ on $k$ to $\BW k$.  
 Moreover, it alters each generator $u_i$ of $E_n$ by a unit, 
 as any degree-$p^n$ isogeny out of $F$ 
 differs by an isomorphism 
 from the multiplication-by-$p$ endomorphism on $F$ 
 (see \cite[1.5-1.6]{Lubin67}).}  
 The coordinate 
 \begin{equation}
  \label{z}
  z \ce \d(y) 
 \end{equation}
 on $\CF$ yields a formal group law $F_z$ over $E_n$ such that $\d \co F_y \to 
 F_z$ is a $\star$-isomorphism.  With respect to $z$, let $b(t) \in E_n \lb t \rb$ be such that 
 \begin{equation}
  \label{bdef}
  g_p^z(t) = t + b(t) 
 \end{equation}
 We will show that this choice of coordinate $z$ gives 
 \begin{equation}
  \label{bhyp}
  b(t) = \sum_{j \geq 1} b_j \, t^{\,j} \hskip 1cm \text{with ~ $b_j \in I^r$} 
 \end{equation}
 and in particular produces the equation 
 \[
  g_p^z(t) \equiv t \md I^r 
 \]
 Note that the formal group laws $F_y$ and $F_z$ coincide modulo $I^{r - 1}$.  
 Thus, by induction and Krull's intersection theorem, we will then obtain in the 
 limit a coordinate $x$ such that $g_p^x(t) = t$, or $l_p^x(t) = f_p^x(t)$, as 
 desired.  
 
 Consider the diagram 
 \begin{equation}
  \label{dsq}
  \begin{tikzpicture}[baseline={([yshift=-10pt]current bounding box.north)}, scale=.88]
   \node (LT) at (0, 3) {$F_y$};
   \node (RT) at (4, 3) {$F_z$};
   \node (LB) at (0, 0) {$\A_n^* t_n^* F_y$};
   \node (RB) at (4, 0) {$\A_n^* t_n^* F_z$};
   \draw [->] (LT) -- node [above] {$\scriptstyle \d$} (RT);
   \draw [->] (LB) -- node [above] {$\scriptstyle \td$} (RB);
   \draw [->] (LT) -- node [left] {$\scriptstyle l_p^{\,y}$} (LB);
   \draw [->] (RT) -- node [right] {$\scriptstyle l_p^{\,z}$} (RB);
  \end{tikzpicture}
 \end{equation}
 where $\td \ce \A_n^* t_n^* \d$ is a $\star$-isomorphism.\footnote{The 
 classifying maps for $F_y / F[p]$ and $F_z / F[p]$ are both $\A_n t_n$ 
 because $F_y$ and $F_z$ are $\star$-isomorphic (Remark \ref{rmk:unique}).}  
 By the unique characterization of $l_p$ in Remark \ref{rmk:l_p}, 
 we have $\td \circ l_p^{\,y} \circ \d^{-1} = l_p^{\,z}$.  Thus 
 the diagram commutes and we get $\td \circ l_p^{\,y}(t) = l_p^{\,z} \circ \d(t)$, or 
 \begin{equation}
  \label{compare}
  \td \circ g_p^y \circ f_p^y(t) = g_p^z \circ f_p^z \circ \d(t) 
 \end{equation}
 We shall compare the two sides of \eqref{compare} modulo $I^r$ to show 
 \eqref{bhyp} and thus complete the induction.  

 The left-hand side of \eqref{compare} can be evaluated modulo $I^r$ as 
 follows.  
 \begin{align*}
  \label{lhs}
  \td \circ g_p^y \circ f_p^y(t) & = \td \big(\,f_p^y(t) + a \circ f_p^y(t) 
                                   \big) & \text{by \eqref{adef}} \st \\
                                 & \equiv \td \big(\,f_p^y(t) + a(t^{\,p^n}) 
                                   \big) & \text{by \eqref{defofrob} and 
                                   \eqref{ahyp}} \\
                                 & = f_p^y(t) + a(t^{\,p^n}) - a \big(\,f_p^y(t) 
                                   + a(t^{\,p^n}) \big) & \text{by \eqref{ddef}} 
                                   \\
                                 & \equiv f_p^y(t) + a(t^{\,p^n}) - a 
                                   \big(\,f_p^y(t) \big) & \text{by 
                                   \eqref{ahyp}} \\
                                 & \equiv f_p^y(t) + a(t^{\,p^n}) - a(t^{\,p^n}) 
                                   & \text{by \eqref{defofrob} and \eqref{ahyp}} 
                                   \\
                                 & = f_p^y(t) 
 \end{align*}
 For the right-hand side of \eqref{compare}, first note that modulo $I^r$ we 
 have 
 \begin{align*}
  \label{rhs}
  f_p^z \circ \d(t) & = \prod_{c \in \CF[p](\tE_n)} \big( z(c) \underset{F_z}{+} 
                      \d(t) \big) & \text{by \eqref{lubinisog}} \st \\
                    & = \prod_c \d \big( y(c) \underset{F_y}{+} t \big) & 
                      \text{by \eqref{z}} \\
                    & = \prod_c \big[ \big( y(c) \underset{F_y}{+} t \big) - 
                      a^{(-p^n)} \big( y(c) \underset{F_y}{+} t \big) \big] & 
                      \text{by \eqref{ddef}} \\
                    & \equiv \prod_c \big( y(c) \underset{F_y}{+} t \big) \\
                    & \quad - \sum_c a^{(-p^n)} \big( y(c) \underset{F_y}{+} t \big) 
                      \prod_{d \neq c} \big( y(d) \underset{F_y}{+} t \big) & 
                      \text{by \eqref{ahyp}} \\
                    & \equiv f_p^y(t) - \sum_c a^{(-p^n)}(t) \, t^{\,p^n - 1} & 
                      \text{by \eqref{lubinisog}, \eqref{ahyp} and 
                      \eqref{defofrob}} \\
                    & = f_p^y(t) - p^n a^{(-p^n)}(t) \, t^{\,p^n - 1} \\
                    & \equiv f_p^y(t) & \text{by \eqref{ahyp} and since 
                      $p \in I$} 
 \end{align*}
 In particular, by \eqref{defofrob}, this gives 
 \begin{equation}
  \label{rhs'}
  f_p^z \circ \d(t) \equiv t^{\,p^n} \md I 
 \end{equation}
 Thus, given $k \geq 2$, if in \eqref{bdef} we have 
 \[
  b(t) = \sum_{j \geq 1} b_j \, t^{\,j} \hskip 1cm \text{with ~ $b_j \in 
  I^{k - 1}$} 
 \]
 then for $k \leq r$, on the right-hand side of \eqref{compare} we have 
 \begin{align*}
  g_p^z \circ f_p^z \circ \d(t) & = f_p^z \circ \d(t) + b \big(\,f_p^z \circ 
                                  \d(t) \big) & \text{by \eqref{bdef}} \\
                                & \equiv f_p^y(t) + b(t^{\,p^n}) \md I^k &
                                  \text{by \eqref{rhs} and \eqref{rhs'}} 
 \end{align*}
 Comparing this to \eqref{lhs}, we get 
 \[
  b(t) \equiv 0 \md I^k 
 \]
 Since $g_p^z$ in \eqref{bdef} is a $\star$-isomorphism, 
 we can proceed by induction on $k$ and obtain 
 \[
  b(t) \equiv 0 \md I^r 
 \]
 which implies 
 \eqref{bhyp}.  

 {\bf Uniqueness} \hskip .18cm Let $\CF / R$ be a deformation of $\CG / k$.  
 Let $x$ and $y$ be two coordinates on $\CF$, both lifting $x_{_G}$ on $\CG$ and both 
 satisfying \eqref{nc}.  Suppose $F_x$ and $F_y$ are in the same 
 $\star$-isomorphism class so that 
 there is a 
 $\star$-isomorphism $\d \co F_x \to F_y$ fitting into a commutative diagram 
 analogous to \eqref{dsq}.  
 \[
  \begin{tikzpicture}[scale=.88]
   \node (LT) at (0, 3) {$F_x$};
   \node (RT) at (4, 3) {$F_y$};
   \node (LB) at (0, 0) {$F_x / F[p]$};
   \node (RB) at (4, 0) {$F_y / F[p]$};
   \draw [->] (LT) -- node [above] {$\scriptstyle \d$} (RT);
   \draw [->] (LB) -- node [above] {$\scriptstyle \td$} (RB);
   \draw [->] (LT) -- node [left] {$\scriptstyle l_p^{\,x}\,=\,f_p^x$} (LB);
   \draw [->] (RT) -- node [right] {$\scriptstyle f_p^y\,=\,l_p^{\,y}$} (RB);
  \end{tikzpicture}
 \]
 Let $\fm$ be the maximal ideal of 
 $R$.  Let $c(t) \in R \lb t \rb$ be such that 
 \begin{equation}
  \label{adef'}
  \d(t) = t + c(t) 
 \end{equation}
 where 
 \[
  c(t) = \sum_{j \geq 1} c_j \, t^{\,j} \hskip 1cm \text{with ~ $c_j \in \fm$} 
 \]
 Since $x$ and $y$ are distinct, there exists $r_0 \geq 2$ such that it is the 
 largest $r$ satisfying 
 \begin{equation}
  \label{r}
  c_j \in \fm^{r - 1} \hskip 1cm \text{for all $j$} 
 \end{equation}
 Modulo $\fm^{r_0}$ we then have 
 \begin{align*}
  f_p^y(t) & = \tilde{\d} \circ f_p^x \circ \d^{-1} (t) \\
           & = f_p^x \circ \d^{-1} (t) + c^{(p^n)} \circ f_p^x \circ \d^{-1} (t) 
             & \text{by \eqref{adef'}} \\
           & \equiv f_p^x \circ \d^{-1} (t) + c^{(p^n)} (t^{\,p^n}) & \text{by 
             \eqref{defofrob} and \eqref{r}} \\
           & \equiv f_p^y(t) + c^{(p^n)} (t^{\,p^n}) & \text{analogous to 
             \eqref{rhs}} 
 \end{align*}
 which is a contradiction.  
\end{proof}

\begin{proof}[Proof of Proposition \ref{prop:nc}] 
 \hskip -.18cm (cf.~\cite[proof of Proposition 2.6.15]{Ando95}) \hskip .18cm We need only 
 show that the coordinate $x$ on $\CF$ constructed in Lemma \ref{klem} 
 satisfies the stronger condition $l_{_H} = f_{_H}$ for any finite $\CH \subset \CF$.  
 As in the proof of existence there, we are reduced to the 
 universal case with $F_\univ \eqqcolon F = F_x$ over $E_n$.  
 
 Given any $\CH \subset \CF$ of degree $p^r$, we will show that the 
 $\star$-isomorphism 
 \begin{equation}
  \label{g_H}
  g_{_H} \co F / H \to \A_r^* t_r^* F 
 \end{equation}
 is the identity by the uniqueness from Lemma \ref{klem}.  Namely, the 
 source and target are in the same $\star$-isomorphism class, and we show 
 that both of them satisfy \eqref{nc}.  That the target does is clear 
 from the proof of 
 Proposition \ref{prop:fun}\,(i).  
 For the source of \eqref{g_H}, let 
 $p^{-1} \CH \ce \{c \in \CF \,|\, p c \in \CH\}$.\footnote{The notation 
 $c \in \CF$ means $[c] \subset \CF$, 
 where $[c]$ is the effective Cartier divisor 
 defined by a section.  
 To be precise, this set-theoretic description 
 defines the subgroup scheme $p^{-1} \CH$ of $\CF$ 
 as a sum of effective Cartier divisors.}  
 It contains both $\CH$ and $\CF[p]$ as subgroups.  We need to show 
 \begin{equation}
  \label{equal}
  l_{_{p^{-1} H / H}}^{\,x_{_H}} = f_{_{p^{-1} H / H}}^{\,x_{_H}} 
 \end{equation}
 where $x_{_H}$ is the coordinate corresponding to the formal group law $F_x / H$.  

 Consider the following commutative diagram, where $x_p$ denotes the coordinate 
 corresponding to $F_x / F[p]$.  The upper rectangle commutes due to the functoriality of 
 the isogeny $f$ under quotient [\ib, Proposition 2.2.6].  The 
 lower rectangle commutes due to the functoriality from Proposition \ref{prop:S} 
 of the $\star$-isomorphisms $g$ under quotient.\footnote{Here we view 
 $F_x / H = F_{x_{_H}}$ 
 as a universal deformation with structure $(\sigma^r, \id)$ \eqref{subsec:univ}.}  
 \[
  \begin{tikzpicture}[scale=.88]
   \node (LT) at (0, 4) {$F_x$};
   \node (RT) at (10, 4) {$F_x / H$};
   \node (LM) at (0, 2) {$F_x / F[p]$};
   \node (RM) at (10, 2) {$F_x / p^{-1}H$};
   \node (LB) at (0, 0) {$(\A_n^x)^* (t_n^{\,x})^* F_x$};
   \node (RB) at (10, 0) {$(\A_n^{x_{_H}})^* (t_n^{\,x_{_H}})^* F_{x_{_H}}$};
   \node (RB1) at (5.95, 0) {$(\A_n^x)^* (t_n^{\,x})^* F_x / (\A_n^x)^* (t_n^{\,x})^* H =$};
   \draw [->] (LT) -- node [above] {$\scriptstyle f_{_H}^x$} (RT);
   \draw [->] (LB) -- node [above] {$\scriptstyle \tilde{f}_{_H}^x$} (RB1);
   \draw [->] (LM) -- node [above] {$\scriptstyle f_{_{p^{-1}H / F[p]}}^{\,x_{_p}}$} 
         (RM);
   \draw [->] (LT) -- node [left] {$\scriptstyle f_p^x$} (LM);
   \draw [double] (LM) -- node [left] {$\scriptstyle g_p^x$} (LB);
   \draw [->] (RT) -- node [right] {$\scriptstyle f_{_{p^{-1}H / H}}^{\,x_{_H}}$} (RM);
   \draw [->] (LT) -- node [above] {$\scriptstyle f_{p^{-1}H}^x$} (RM);
   \draw [->] (RM) -- node [right] {$\scriptstyle g_{_{p^{-1}H / H}}^{x_{_H}}$} (RB);
  \end{tikzpicture}
 \]
 Note that $p^{-1}\CH / \CF[p] \cong \CH$.  In the lower rectangle, 
 $g_p^x = \id$ and hence $f_{_{p^{-1}H / F[p]}}^{\,x_{_p}} = \tilde{f}_{_H}^x$.  This 
 forces the isomorphism $g_{_{p^{-1}H / H}}^{x_{_H}}$ to be the identity, and \eqref{equal} follows.  
\end{proof}

\begin{rmk}
 In \cite[\S 3.3]{ho}, for the purpose of studying 
 Hecke operators in elliptic cohomology, 
 we showed the existence of an analogue of Ando's coordinate.  It is conceptually 
 different from the norm-coherent coordinates here.  
 Note that there the base change is not along a {\em local} homomorphism 
 (cf.~\cite[\S 4, footnote]{p3} and see \eqref{subsec:morefun} below).  
\end{rmk}

\begin{ex}
 Let $k = \BF_{p^2}$ and $G$ be the formal group law of a supersingular elliptic 
 curve over $k$.  We choose a curve such that its $p^2$-power Frobenius 
 endomorphism coincides with the map of multiplication by $(-1)^{\,p - 1} p$, as 
 in \cite[3.24]{ho}.  We then have $l_p = [-p]$ if $p = 2$ by rigidity.  
 
 Let $\bE$ be the Morava E-theory associated to $G / k$, and choose a {\em 
 preferred $\CP_N$-model} for $\bE$ in the sense of [\ib, Definition 3.29].  
 In particular, there is a chosen coordinate $u$ on the universal deformation of 
 $\CG / k$.  Given [\ib, 3.28], the cotangent map along $f_p^u$ is 
 multiplication by $p$.  Thus, by the criterion \eqref{nc}, $u$ 
 cannot be norm-coherent if $p = 2$.  
\end{ex}

\begin{subsec}{\tt{More functoriality of norm coherence}}
 \label{subsec:morefun}
 We continue the discussion in \eqref{subsec:fun} with varying $G / k$.  
 Let $\Coord \co \FG_\isog \to \Set$ be the ``wide functor'' 
 \[
  \CG / k \mapsto \{\text{coordinates on $\CG$}\} \subset \CO_G 
 \]
 in the following sense: 
 given the diagram \eqref{widemor}, $\Coord$ is 
 contravariant along the right square and covariant along 
 the left square.  More specifically, $\Coord$ is 
 contravariant with respect to base change 
 $\Spf k \to \Spf k'$ and pullback along an isomorphism over $k$, 
 hence contravariant with respect to any morphism in the subcategory $\FG_\iso$.  
 On the other hand, given an isogeny $\CG \to \CG'$ over $k$ 
 of degree $p^r$, 
 any coordinate $x$ on $\CG$ determines a unique coordinate on $\CG^{(p^r)}$ 
 which pulls back along $\Frob^r$ to $x^{\,p^r}$; this coordinate on $\CG^{(p^r)}$ then 
 corresponds to one on $\CG'$ via the isomorphism between the two formal groups.  
 Thus $\Coord$ is also covariant with respect to any morphism in the subcategory $\FG_\isog(k)$.  

 Let $\NCoh \co \FG_\isog \to \Set$ be the wide functor 
 \[
  \CG / k \mapsto \{\text{norm-coherent coordinates on $\CF_\univ(\CG) / E_n$}\} 
 \]
 where $\CF_\univ(\CG)$ is a choice of universal deformation of $\CG$ as in 
 \eqref{subsec:univ}.  
 Its ``wideness,'' in the same sense as above, follows from 
 Proposition \ref{prop:fun} and the discussion in \eqref{subsec:univ}.  
\end{subsec}

\begin{thm}
 \label{thm:ideal}
 The natural transformation $\rho \co \NCoh \to \Coord$ of wide functors by 
 restricting a coordinate on $\CF_\univ(\CG)$ to $\CG$ is an isomorphism.  
 Moreover, it satisfies Galois descent: given $\CG / k$ in $\FG_\isog$ 
 and a Galois extension $K / k$, the following commutes, where the 
 vertical maps take fixed points under the Galois action.  
 \[
  \begin{tikzpicture}[scale=.88]
   \node (LT) at (0, 3) {$\NCoh(\CG \times_k K)$};
   \node (RT) at (4, 3) {$\Coord(\CG \times_k K)$};
   \node (LB) at (0, 0) {$\NCoh(\CG)$};
   \node (RB) at (4, 0) {$\Coord(\CG)$};
   \draw [->] (LT) -- node [above] {$\scriptstyle \sim$} (RT);
   \draw [->] (LB) -- node [above] {$\scriptstyle \sim$} (RB);
   \draw [->] (LT) -- (LB);
   \draw [->] (RT) -- (RB);
  \end{tikzpicture}
 \]
 Moreover, this diagram is natural in $\CG / k$ and $K / k$.  
\end{thm}
\begin{proof}
 On each object in $\FG_\isog$, the natural transformation $\rho$ is an isomorphism by 
 Proposition \ref{prop:nc}, and the descent is clear since the condition 
 \eqref{l=f} of norm coherence is stable under Galois actions.  
 Each of the naturality properties is straightforward to check.  
\end{proof}

\section{Norm coherence and \Hoo\,complex orientations}
\label{sec:ortn}

\begin{subsec}{\tt{Introduction}}
 \label{subsec:intro}
 Given a Morava E-theory spectrum $\bE$, consider its complex orientations, or, 
 more precisely, homotopy multiplicative maps $MU\<0\> \to \bE$.  A necessary 
 and sufficient condition for such an orientation to be \Hoo\,is that its 
 corresponding coordinate on the formal group of $\bE$ is norm-coherent.  Ando 
 showed this for E-theories associated to the Honda formal groups over $\BF_p$ 
 \cite[Theorem 4.1.1]{Ando95}.  There, the norm-coherence condition \eqref{l=f} 
 boils down to the identity $[p] = \, f_p$ (cf.~\eqref{nc} and \eqref{ex:Honda}).  Moreover, he established the existence and uniqueness of 
 coordinates, hence orientations, with the desired property [\ib, Theorem 
 2.6.4].  
 
 In fact, to show that norm coherence is necessary and sufficient for \Hoo 
 orientations, Ando's proof does not depend on the choice of the formal groups 
 being the Honda formal groups (see [\ib, Lemma 4.4.4]).  However, his setup does 
 require them be defined over $\BF_p$ so that the relative $p^r$-power Frobenius 
 is an endomorphism for every $r \geq 0$ (cf.~[\ib, Proposition 2.5.1] and 
 Remark \ref{rmk:endo}).  
 
 With results in sufficient generality about level structures on formal groups 
 from \cite{Str97}, Ando, Hopkins, and Strickland extended the applicability of 
 the above condition for \Hoo\,orientations: $MU\<0\>$ generalizes to 
 $MU\<2 k\>$, $k \leq 3$, and $\bE$ generalizes to any even periodic 
 \Hoo\!-ring spectum whose zeroth homotopy is a $p$-regular admissible local 
 ring with perfect residue field of characteristic $p$ and whose formal group is 
 of finite height \cite[Proposition 6.1]{AHS04}.  They did this by first 
 reformulating Ando's condition so that in particular it applies to E-theories 
 associated to formal groups over any perfect field of positive characteristic 
 [\ib, Proposition 4.13].  
 
 Based on this general condition, they established the existence and uniqueness 
 of \Hoo\,$MU\<6\>$-orientations for \Hoo\,elliptic spectra, called the sigma 
 orientations, from corresponding norm-coherent cubical structures of elliptic 
 curves [\ib, Proposition 16.5].  However, when the elliptic spectrum represents 
 an E-theory associated to the formal group of a supersingular elliptic curve, 
 such an orientation does not factor through $MU\<4\>$ due to obstruction from 
 Weil pairings (see \cite[proof of Theorem 1.4]{Wpair}).  Thus, in this case, we 
 cannot deduce the existence and uniqueness of \Hoo\,$MU\<2 k\>$-orientations for 
 $0 \leq k \leq 2$ from the sigma orientation.  
\end{subsec}

\begin{subsec}{\tt{Set-up}}
 \label{subsec:setup''}
 Let $\bE$ be a Morava E-theory spectrum, with 
 $\GE = \CF_\univ(\CG)$ for some $\CG / k$ whose group law is as in \eqref{subsec:setup'}.  We 
 will show the existence and uniqueness of \Hoo\,$MU\<0\>$-orientations 
 for $\bE$ by combining Proposition \ref{prop:nc} 
 with Ando, Hopkins, and Strickland's 
 condition for \Hoo orientations.  Specifically, we need only check that their 
 \cite[4.14]{AHS04} and our definition \eqref{l=f} for norm 
 coherence agree.  
\end{subsec}

\begin{subsec}{\tt{Descent for level structures on Lubin-Tate formal groups}}
 \label{subsec:descent}
 We carry out the needed comparison by recalling the canonical 
 descent data for level structures on $\GE = \CF_\univ(\CG)$ from 
 [\ib, Part 3].  Since $\CG$ is over $k$ of characteristic $p$, the finite 
 subgroups $\CH$ of $\GE$ 
 must be of degree $p^r$ for some $r \geq 0$.  
 
 \begin{defn}[{cf.~[\ib, Definitions 3.1, 9.9, Proposition 10.10\,(i), 12.5]}]
  \label{def:level}
  Let $A$ be an ``abstract'' finite abelian group 
  of order $p^r$.  Let 
  $\SE = \Spf \PE$ and 
  $T = \Spf R$ with $R$ as in \eqref{subsec:setup'}.  
  Let $i \co T \to \SE$ be a 
  morphism of formal schemes, 
  faithfully flat and locally of finite presentation, 
  which classifies a deformation of $\CG / k$ to $R$.  
  Write $\CA_T$ for the constant formal group scheme of $A$ over $T$.  
  A morphism 
  \[
   \ell \co \CA_T \to i^* \GE 
  \]
  of formal groups over $T$, 
  equivalent to a group homomorphism $\phi_{_\ell} \co A \to i^* \GE(T)$, 
  is a {\em level $A$-structure on $\GE$} 
  if the effective Cartier divisor $\CD_\ell \ce \sum_{a \in A} [\phi_{_\ell}(a)]$ of 
  degree $p^r$ 
  is a subgroup of $i^* \GE$.  
 \end{defn}
 
 \begin{rmk}
  \label{rmk:level}
  Note that a level $A$-structure $\ell$ on $\GE$ uniquely corresponds 
  to a finite subgroup $\CH = \CD_\ell$, which is different from the scheme-theoretic image 
  of $\CA_T$ under $\ell$ (the latter automatically a subgroup, 
  but possibly of smaller degree).  
  Automorphisms of $A$ correspond to automorphisms of $\CH$ 
  (cf.~[\ib, Definition 3.1\,(3)]).  
 \end{rmk}
 
 \begin{defn}[{cf.~[\ib, Definition 3.9, Remark 3.12]}]
  \label{def:psil}
  Let $\ell \co \CA_{\,\Spf R} \to i^* \GE$ be a level $A$-structure on $\GE$ as above.  
  Define $\psi_\ell^{_\bE} \co \PE \to R$ to be the composite 
  \[
   \PE \xrightarrow{~D_{p^r}} \PE^{(B\Sigma_{p^r})_+} 
   \to \PE^{(B\Sigma_{p^r})_+} / I_\tr \xrightarrow{\,\A_r} R 
  \]
  where the power operation $D_{p^r}$ arises from the \Hoo\!-ring structure of $\bE$ 
  \eqref{subsec:EH}, 
  $I_\tr$ is the ideal generated by the images of transfers 
  from proper subgroups of $\Sigma_{p^r}$, and $\A_r$ classifies the subgroup 
  of $i^* \GE$ corresponding to $\ell$ 
  (\ref{rmk:level}, \ref{subsec:E}).  
 \end{defn}

 \begin{rmk}
  \label{rmk:psil}
  In the presence of a level structure as in Definition \ref{def:psil}, 
  the structure morphism $i$ of $T$ over $\SE$ in Definition \ref{def:level} 
  is given by the classifying map 
  \[
   \A \co A_0 \xrightarrow{\,s_r} A_r \xrightarrow{\,\A_r} R 
  \]
  from Propositions \ref{prop:LT} and \ref{prop:S}, 
  while $\psi_\ell^{_\bE}$ is precisely the classifying map 
  \[
   \A' \co A_0 \xrightarrow{\,t_r} A_r \xrightarrow{\,\A_r} R 
  \]
  (cf.~\cite[Theorem B]{cong} for the identification with $t_r$).  
 \end{rmk}
 
 \begin{defn}[{cf.~\cite[3.13-3.15]{AHS04}}]
  \label{def:psig}
  Let $\bF \ce \bE^{X_+}$ and $\Bf \co \bE \to \bF$ be the natural map of 
  \Hoo\!-ring spectra.  
  Given any level $A$-structure $\ell \co \CA_T \to i^* \GE$ on $\GE$, 
  let $\ell'$ be the unique level $A$-structure on $\GF$ 
  induced by $\Bf$,\footnote{{\em Level 
  $A$-structures on $\GF$} are 
  defined analogously to those on $\GE$ as in Definition \ref{def:level}.} 
  so that the following diagram commutes 
  (with all but the front-left and back-left squares cartesian).  
  \[
   \begin{tikzpicture}[scale=.88]
    \node (AA) at (2, 6) {$\CA_{T'}$};
    \node (AB) at (6, 6) {$i'^* \GF$};
    \node (AC) at (10, 6) {$\GF$};
    \node (BA) at (0, 4) {$\CA_T$};
    \node (BB) at (4, 4) {$i^* \GE$};
    \node (BC) at (8, 4) {$\GE$};
    \draw [->] (AA) -- node [above] {$\scriptstyle \ell'$} (AB);
    \draw [->] (AB) -- (AC);
    \draw [->] (BA) -- node [above] {$\scriptstyle \quad\ell$} (BB);
    \draw [->] (BB) -- (BC);
    \draw [->] (AA) -- (BA);
    \draw [->] (AB) -- (BB);
    \draw [->] (AC) -- node [right] {$\scriptstyle \Gf$} (BC);
    \node (AA') at (2, 3) {$T'$};
    \node (AB') at (6, 3) {$i^* \SF$};
    \node (AC') at (10, 3) {$\SF$};
    \node (BA') at (0, 1) {$T$};
    \node (BB') at (4, 1) {$T$};
    \node (BC') at (8, 1) {$\SE$};
    \draw [double] (AA') -- (AB');
    \draw [->] (AB') -- node [above] {$\scriptstyle i'\quad$} (AC');
    \draw [double] (BA') -- (BB');
    \draw [->] (BB') -- node [above] {$\scriptstyle i$} (BC');
    \draw [->] (AA') -- (BA');
    \draw [->] (AB') -- (BB');
    \draw [->] (AC') -- node [right] {$\scriptstyle \Sf$} (BC');
    \draw [->] (AA) -- (AA');
    \draw [->] (AB) -- (AB');
    \draw [->] (AC) -- (AC');
    \draw [->] (BA) -- (BA');
    \draw [->] (BB) -- (BB');
    \draw [->] (BC) -- (BC');
   \end{tikzpicture}
  \]
  Let $\psi_{\ell'}^{_\bF} \co T' \to \SF$ be the morphism 
  analogous to $\psi_\ell^{_\bE}$ in Definition \ref{def:psil}, obtained 
  by the naturality of power operations on $\bE^0(X)$.  
  Define $\psi_\ell^{_{\bF / \bE}}$ to be the unique $T$-morphism 
  which fits into the following commutative diagram.  
  \begin{equation}
   \label{frob'}
   \begin{tikzpicture}[baseline={([yshift=-10pt]current bounding box.north)}, scale=.88]
    \node (LT) at (0, 3) {$\SF$};
    \node (MLT) at (4, 3) {$i^* \SF$};
    \node (MRT) at (8, 3) {$(\psi_\ell^{_\bE})^* \SF$};
    \node (RT) at (12, 3) {$\SF$};
    \node (LB) at (0, 0) {$\SE$};
    \node (MLB) at (4, 0) {$T$};
    \node (MRB) at (8, 0) {$T$};
    \node (RB) at (12, 0) {$\SE$};
    \draw [<-] (LT) -- (MLT);
    \draw [->] (MLT) -- node [above] {$\scriptstyle \psi_\ell^{_{\bF / \bE}}$} (MRT);
    \draw [->] (MRT) -- (RT);
    \draw [->] (LT) -- node [left] {$\scriptstyle \Sf$} (LB);
    \draw [->] (MLT) -- (MLB);
    \draw [->] (MRT) -- (MRB);
    \draw [->] (RT) -- node [right] {$\scriptstyle \Sf$} (RB);
    \draw [<-] (LB) -- node [above] {$\scriptstyle i$} (MLB);
    \draw [double] (MLB) -- (MRB);
    \draw [->] (MRB) -- node [above] {$\scriptstyle \psi_\ell^\bE$} (RB);
    \draw [->] (MLT) to [out = 30, in = 150] node [above] {$\scriptstyle \psi_{\ell'}^\bF$} (RT);
    \node at (8.5, 2.5) {$\lrcorner$}; 
    \node at (3.5, 2.5) {$\llcorner$}; 
   \end{tikzpicture}
  \end{equation}
  In particular, when $\bF = \bE^{(\BC \BP^\infty)_+}$, 
  write $\psi_\ell^{_{\bF / \bE}}$ as 
  \[
   \psi_\ell^{_{\CG / \bE}} \co i^* \GE \to (\psi_\ell^{_\bE})^* \GE 
  \]
 \end{defn}
 
 \begin{rmk}
  \label{rmk:psig}
  Let $\bF = \bE^{(\BC \BP^\infty)_+}$.  
  When $A = \BZ / p$, the diagram \eqref{frob'} lifts \eqref{frob}.  
  More generally, let $\CH \subset i^* \GE$ 
  correspond to $\ell$ as in Remark \ref{rmk:level}.  
  Comparing \eqref{frob'} to the universal example \eqref{univex} and 
  Remark \ref{rmk:psil}, we see that $\psi_\ell^{_{\CG / \bE}}$ is precisely 
  the isogeny $l_{_H}$ from \eqref{subsec:l_H} if we assume without loss 
  of generality that the $\star$-isomorphism \eqref{wlog} is the identity.  
 \end{rmk}
\end{subsec}

\begin{subsec}{\tt{Norm maps}}
 In view of \cite[Theorem 3.25]{AHS04}, 
 we have compared above the ingredients that constitute descent data for 
 level structures on $\GE$---level structures $\ell$, 
 classifying maps $i$ and $\psi_\ell^{_\bE}$, 
 isogenies $\psi_\ell^{_{\CG / \bE}}$---with 
 corresponding terms from the earlier sections of this paper.  
 There is one more and key ingredient that goes into the condition 
 [\ib, 4.14] for \Hoo\,$MU\<0\>$-orientations.  
 
 \begin{defn}[{cf.~[\ib, Definitions 10.1, 10.9]}]
  Let $\psi \co \CG \to \CG'$ be an isogeny of formal groups with kernel $\CK$.  
  Let $\mu, \pi \co \CG \times \CK \to \CG$ be the multiplication, 
  projection maps, and $q \co \CG \to \CG / \CK$ be the quotient map, as in 
  \eqref{subsec:quotient}.  
  Define $N_\psi \co \CO_G \to \CO_{G'}$ to be the horizontal composite 
  \begin{equation}
   \label{factor''}
   \begin{tikzpicture}[baseline={([yshift=-10pt]current bounding box.north)}, scale=.88]
    \node (L) at (0, 0) {$\CO_G$};
    \node (ML) at (2, -1) {$\CO_{G \times K}$};
    \node (MR) at (4, 0) {$\CO_{G / K}$};
    \node (R) at (5.3, 0.06) {$\xrightarrow{\sim} \CO_{G'}$};
    \node (T) at (4, -2) {$\CO_G$};
    \node (B) at (4, -4) {$\CO_{G \times K}$};
    \draw [->] (MR) -- node [right] 
          {$\scriptstyle q^*$} (T);
    \draw [dashed, ->] (L) -- (MR);
    \draw [->] (L) -- node [left] 
          {$\scriptstyle \text{\raisebox{-.5cm}{$\mu^*$} \hskip -.2cm}$} (ML);
    \draw [->] (ML) -- node [left] 
          {$\scriptstyle \text{\raisebox{-.4cm}{$\Norm_{\pi^*}$ \hskip -.3cm}}$} (T);
    \draw [transform canvas={xshift=0.5ex}, ->] (T) -- node [left] 
          {$\scriptstyle \mu^*\,$} (B);
    \draw [transform canvas={xshift=-0.5ex}, ->] (T) -- node [right] 
          {$\scriptstyle ~\pi^*$} (B);
   \end{tikzpicture}
  \end{equation}
  where the vertical maps exhibit $\CO_{G / K}$ as an equalizer, 
  $\Norm_{\pi^*}$ sends $a$ to the determinant of multiplication by $a$ 
  on $\CO_{G \times K}$ as a finite free $\CO_G$-module via $\pi^*$, 
  and the factorization through $\CO_{G / K}$ 
  was shown, \eg, in \cite[Theorem 19]{Str97}.  
 \end{defn}
 
 \begin{rmk}
  \label{rmk:norm'}
  Since $q \circ \mu = q \circ \pi$, 
  we have $\Norm_{\pi^*} \circ \mu^* = q^* \circ \Norm_{q^*}$ 
  (by an argument similar to the proof of the factorization 
  mentioned above).  Thus the dashed arrow in \eqref{factor''} is 
  $\Norm_{q^*}$ by uniqueness from the universal property of an equalizer.  
  
  Suppose that the isogeny $\psi$ is over 
  a field $k$ of characteristic $p$, and is hence of degree $p^r$ for some $r \geq 0$.  
  Comparing [\ib, Theorem 19\,(i)] and Lemma \ref{lem:norm}, 
  we see that $N_\psi$ is precisely the map $\Lambda_\psi^*$ in \eqref{factor'}.  
 \end{rmk}
 
 \begin{rmk}
  \label{rmk:norm}
  Let $\psi$ be the isogeny $\psi_\ell^{_{\CG / \bE}} 
  \co i^* \GE \to (\psi_\ell^{_\bE})^* \GE$ over $R$ 
  from Definition \ref{def:psig}.  
  Let $x$ be any coordinate on $i^* \GE$.  
  In view of Remark \ref{rmk:level}, we have from \cite[10.11]{AHS04} that 
  \[
   \psi^* N_\psi (x) = \prod_{a \in A} T_a^* (x) 
   = \prod_{Q \in \CH(R)} \big( x \underset{i^* G_{_\bE}}{+} x(Q) \big) 
  \]
  where $T_a \co i^* \GE \to i^* \GE$ translates any $R$-point $P$ on $i^* \GE$ to 
  $P + Q$, with $Q = \phi_{_\ell}(a)$.  
  Comparing this to \eqref{x_H}, with $\CF = i^* \GE$, we see that 
  \[
   \psi^* N_\psi (x) = f_{_H}^* (x_{_H}) \hskip .5cm 
  \]
  Now, given any coordinate 
  $s$ on $\GE$, the condition [\ib, 4.14] states that 
  \[
   (\psi_\ell^{_\bE})^* s = N_{\psi_\ell^{_{\CG / \bE}}} (i^* s) \hskip -.5cm 
  \]
  Pulling this back along $\psi_\ell^{_{\CG / \bE}}$ and writing $x \ce i^* s$, 
  we get an equivalent identity 
  \begin{equation}
   \label{l=f'}
   l_{_H}^* (x') = f_{_H}^* (x_{_H}) 
  \end{equation}
  where $l_{_H} = \psi_\ell^{_{\CG / \bE}}$ from Remark \ref{rmk:psig}, 
  and $x' \ce (\psi_\ell^{_\bE})^* s = \A_r^* t_r^* s$ from Remark \ref{rmk:psil}.  
  In view of (\ref{l_H}, \ref{fgl}), 
  we see that \eqref{l=f'} is equivalent to \eqref{l=f}.  
  This shows that [\ib, 4.14] and our norm-coherence condition agree 
  (cf.~\eqref{norm''}).  
 \end{rmk}
\end{subsec}

\begin{cor}
 \label{cor}
 Let $\bE$, $\GE$, and $\CG$ be as in \eqref{subsec:setup''}.  
 Given any coordinate $x_{_G}$ on $\CG$, there exists a unique coordinate $x$ on 
 $\GE$ lifting $x_{_G}$ such that its corresponding 
 $MU\<0\>$-orientation for $\bE$ is \Hoo.  
\end{cor}
\begin{proof}
 In view of Remark \ref{rmk:norm}, 
 the corollary follows from Proposition \ref{prop:nc}.  
 In particular, as $p$ is not a zero-divisor in $\PE$, 
 we may apply \cite[Proposition 6.1]{AHS04} 
 for \Hoo\,$MU\<2 k\>$-orientations with $k = 0$ 
 (cf.~[\ib, discussion following 1.6]).  
\end{proof}

\renewcommand\refname{}
\newcommand{\AX}[1]{\href{http://arxiv.org/abs/#1}{arXiv:#1}}
\renewcommand{\MR}[1]{\href{http://www.ams.org/mathscinet-getitem?mr=#1}{MR#1}}
\wt{.}

\end{document}